\documentclass[11pt]{amsart}
\usepackage[english]{babel}
\usepackage{graphicx}
\usepackage{subcaption}
\usepackage[utf8]{inputenc}
\usepackage[margin=1in]{geometry}
\usepackage{amsthm,amssymb,amsmath}
\usepackage{amsfonts}
\usepackage{indentfirst}
\usepackage{mathtools}
\usepackage{hyperref}
\usepackage{amsrefs}

\DeclareMathOperator\supp{supp}
\DeclareMathOperator\PV{PV}
\DeclareMathOperator\diver{div}
\DeclareMathOperator\Co{Co}

\newtheorem{lemma}{Lemma}[section]

\newtheorem{proposition}{Proposition}[section]
\newtheorem{theorem}{Theorem}[section]
\newtheorem{corollary}{Corollary}[section]

\theoremstyle{definition}

\theoremstyle{definition}
\newtheorem{definition}{Definition}[section]

\theoremstyle{remark}
\newtheorem*{remark}{Remark}

\begin{document}
\title{Existence and Regularity for Vortex Patch Solutions of the 2D Euler Equations}
\author{R\u{a}zvan-Octavian Radu}
\address{Departemnt of Mathematics, Princeton University, Princeton, NJ 08544}
\email{rradu@princeton.edu}
\date{June 27, 2020}
\keywords{Euler equations, vortex patches, active scalars}
\subjclass[2010]{35Q31, 35Q35, 76B47}
\begin{abstract}
In \cite{BC}, Bertozzi and Constantin formulate the vortex patch problem in the level-set framework and prove a priori estimates for this active scalar equation. By extending the tools used to prove these estimates, we construct solutions and show propagation of higher H\"older regularity. This constitutes a proof of the regularity of vortex patches, carried out solely in the level-set framework. 
\end{abstract}
\maketitle

\section{Introduction}

A vortex patch is an example of a Yudovich (\cite{Y}) solution of the two-dimensional incompressible Euler equations. It is given by an initial vorticity which is constant in a simply connected, open and bounded domain $D_0 \subset \mathbb{R}^2$ (the patch) and zero outside of this domain. The resulting velocity transports the patch and preserves its topological properties as well as the step-function characteristic of the vorticity. Given a patch with smooth initial boundary, it is natural to ask whether the boundary remains smooth for all time (\cite{M}). The problem was settled in the affirmative by Chemin in \cite{Ch}. More precisely, using methods of paradifferential calculus, it is shown in \cite{Ch} that if the boundary of the patch is $C^{k, \mu}$ at the initial time, then it remains so for all time. Soon afterwards, a different proof of this result was given by Bertozzi-Constantin in \cite{BC}. Their approach builds on a local existence result for the contour dynamics equation of the boundary (\cite{B}, see also chapter 8 of \cite{MB}) by proving a priori estimates in the level-set framework. Later, yet another approach was presented by Serfati in \cite{S} (see also \cite{BK}) . Recently, the regularity of vortex patches was also studied in Sobolev spaces in \cite{CS}.

Following \cite{BC}, we formulate the vortex patch problem in the level-set framework. More precisely, we consider the following active scalar equation:

Let $\phi: \mathbb{R}^2 \times [-T, T] \rightarrow \mathbb{R}$ be a solution of equations 
\begin{equation} \label{eq1.1}
\partial_t \phi + v \cdot \nabla \phi = 0, 
\end{equation}
\begin{equation} \label{eq1.2}
\phi(x,0) = \phi_0 (x),
\end{equation}
\begin{equation} \label{eq1.3}
v(x,t) = \omega_0 \int_{\mathbb{R}^2} K(x - y) H(\phi(y,t)) dy,
\end{equation}
where $\omega_0 \in \mathbb{R}$ is nonzero, the kernel is given by $K(x) = \frac{1}{2\pi} |x|^{-2} \begin{bmatrix}
-x_2 \\
x_1
\end{bmatrix},$ and $H$ is the Heaviside function: $H(\phi) = 1$ if $\phi \geq 0$ and $H(\phi) = 0$ otherwise.  The patch is defined by $D = D(t)$, 
\begin{equation*}
D = \{ x \in \mathbb{R}^2 \mid \phi(x,t) > 0 \}.
\end{equation*}
We assume that $D(0) = D_0 =  \{ x \in \mathbb{R}^2 \mid \phi_0(x) > 0 \}$ is bounded and simply connected, that  $\phi_0 \in C^{k, \,\mu} (\mathbb{R}^2)$, $\mu \in (0,1)$ and $k \in \mathbb{Z}_+$, and that 
\begin{equation} \label{eq1.4}
\inf_{x \in \partial D_0} |\nabla \phi_0(x)| \geq m > 0,
\end{equation}
where $C^{k, \mu}$ is the Banach space of $k$-times continuously differentiable functions with bounded, $\mu$-H\"older continuous partial derivatives and endowed with the standard norm, and $\partial D_0$ denotes the boundary of the domain $D_0$. Condition \eqref{eq1.4}, in view of the inverse mapping theorem, assures that the boundary of the initial patch is $C^{k, \, \mu}$. We will also use the following notation: 
\begin{equation*}
|\nabla \phi|_{\inf} = \inf \{|\nabla \phi (x)| \mid x \in \phi^{-1}(\{0\})\} ,
\end{equation*}
and for the H\"older semi-norms: 
\begin{equation*}
|\nabla \phi|_\mu = \sup_{x \neq y} \frac{|\nabla \phi (x) - \nabla \phi(y)|}{|x - y|^\mu}.
\end{equation*}

For the sake of exposition, let us introduce the following definitions. 

\begin{definition}
We say $\phi_0:\mathbb{R}^2 \rightarrow \mathbb{R}$ is a $C^{k, \, \mu}$ initial scalar ($0 < \mu < 1$, $k \in \mathbb{Z}_+$) if $\phi_0 \in C^{k, \, \mu}(\mathbb{R}^2)$, $|\nabla \phi_0|_{\inf} \geq m > 0$, and $D_0$ is simply connected and bounded. 
\end{definition}

\begin{definition}
We say that $\phi:\mathbb{R}^2 \times \mathbb{R} \rightarrow \mathbb{R}$ is a solution to the vortex patch equations with $C^{k, \, \mu}$ initial scalar if it solves equations \eqref{eq1.1}-\eqref{eq1.3} and $\phi_0$ is a $C^{k, \, \mu}$ initial scalar. 
\end{definition}

The results of this paper are best regarded in contrast with the proof of Bertozzi-Constantin. In \cite{B} (see also section 8.3 of \cite{MB}), using the Picard-Lindel\"of theorem on the Banach spaces $C^{k, \mu}(\mathbb{S}^1)$, $k \geq 1$, $0 < \mu < 1$, it is proven that local existence and uniqueness of solutions hold in these spaces for the contour dynamics equation (\cite{ZHR})
\begin{equation*}
\frac{d}{dt}z(\alpha, t) = - \frac{\omega_0}{2 \pi} \int_0^{2\pi} \ln|z(\alpha, t) - z(\alpha', t)| \partial_\alpha z(\alpha', t)  d\alpha',
\end{equation*}
which describes the motion of the boundary $\partial D(t) = z(\cdot, t) \in C^{k ,\mu}(\mathbb{S}^1)$ of the patch. Moreover, in this framework still, higher H\"older norms are shown to be controlled by lower ones, which means that it is enough to show a priori estimates for the $C^{1, \mu}$ norm in order to conclude that higher H\"older regularity for the patch is preserved for all time. And indeed, these estimates are obtained in \cite{BC}, by considering the level-set framework of equations \eqref{eq1.1}-\eqref{eq1.3}. 

In this paper, we make use of the ideas involved in the a priori estimates of \cite{BC} to show that if $\phi_0$ is a $C^{1, \, \mu}$ initial scalar (in the sense of the definition above), then there exists a global solution of equations \eqref{eq1.1}-\eqref{eq1.3} which is $C^{1, \, \mu}$ for all time. Then, we show that if $\phi_0$ is, moreover, a $C^{k, \, \mu}$ initial scalar, the boundary of the patch remains $C^{k, \, \mu}$ for all time. Both the construction of $C^{1, \mu}$ solutions and the propagation of the higher H\"older regularity are carried out in the level-set framework, and, thus, represent a new proof of the regularity of vortex patches. 

More precisely, the main theorems of this paper are the following:

\begin{theorem} \label{thm1}
There exists a continuously differentiable solution $\phi: \mathbb{R}^2 \times \mathbb{R} \rightarrow \mathbb{R}$ of the vortex patch equations with $C^{1, \, \mu}$ initial scalar $\phi_0$ such that $\phi(\cdot, t) \in C^{1, \, \mu} (\mathbb{R}^2)$ and $|\nabla \phi(\cdot, t)|_{\inf}>0$ for all times $t \in \mathbb{R}$. 
\end{theorem}

\begin{theorem} \label{thm2}
Let $\phi_0 \in C^{k, \, \mu}(\mathbb{R}^2)$, $k \geq 1$, $0 < \mu <1$, be such that $D_0$ is bounded and simply connected, and $|\nabla 
\phi_0|_{\inf} \geq m > 0$. Let $\omega_0 \in \mathbb{R}$, and define the initial vorticity $\omega^0(x) = \omega_0 H(\phi_0(x))$. Then, the unique Yudovich solution $\omega(x,t)$ is the characterisitc function of a bounded and simply connected domain with $C^{k, \, \mu}$ boundary for all time.  
\end{theorem}

We now briefly describe the strategy. To obtain a $C^{1, \, \mu}$ solution of the vortex patch equations \eqref{eq1.1}-\eqref{eq1.3}, we first mollify the Heaviside function and obtain a sequence of smooth approximations $H_n$. For each $n \in \mathbb{Z}_+$, we define the initial vorticities $ \omega_0^n (x) = \omega_0 H_n(\phi_0(x))$, which are $C^{1, \, \mu}$ and compactly supported. It is well-known that the 2D vorticity equation has a global solution for such initial vorticities (see, for example, chapter 4 of \cite{MB}). Using the particle-trajectory mappings $X_n$, we find solutions $\phi^n$ to equations which approximate \eqref{eq1.1}-\eqref{eq1.3}. Following and adapting the methods of \cite{BC}, we obtain uniform bounds for the $L^\infty$ norms of the gradients of the velocities, $\nabla v_n$, and for the H\"older norms of the gradients of the solutions, $\nabla \phi^n$. Then, using arguments similar to those presented in \cite{MB} for the existence of Yudovich weak solutions, we find a subsequence of particle trajectories which converges uniformly on compact sets. Moreover, using the estimates we have obtained, we can apply the Arzel\`a-Ascoli theorem to find a further subsequence of $\nabla \phi^n$ which also converges uniformly on compact sets. Finally, we check that the obtained quantities solve the equations \eqref{eq1.1}-\eqref{eq1.3}. Since the methods are essentially based on the a priori estimates of \cite{BC}, $C^{1, \, \mu}$ regularity of these solutions is automatic, and so theorem \ref{thm1} is proven. 

For the higher regularity case, we consider parametrizations of the boundaries of the approximate patches $z_n:S^1 \times \mathbb{R} \rightarrow \mathbb{R}^2$. Using the facts that the tangent vectors are at each point perpendicular to the gradients $\nabla \phi^n$ and that the boundaries are transported by the particle trajectories $X_n$, we are able to obtain expressions for the derivatives of $z_n$, which lend themselves to an inductive argument. This yields uniform bounds for the H\"older norms in terms of certain quantities which are natural extensions of those which lie at the heart of the arguments in \cite{BC}. Then, we show that a subsequence converges to the boundary of the patch given by theorem \ref{thm1}, thus implying that the regularity of the patch persists and proving theorem \ref{thm2}. 

\section{Existence and Regularity of Solutions of the Approximate Equations}

In this section, we prove that, given a $C^{k, \, \mu}$ initial scalar $\phi_0$, there exist $C^{k, \, \mu}$ solutions to a set of equations which approximate equations \eqref{eq1.1}-\eqref{eq1.3}.

Let $\rho: \mathbb{R} \rightarrow \mathbb{R}$ be a standard mollifier, supported in $[-1,1]$, and with $1 \geq \rho(x) \geq 0$, $\int \rho = 1$. Define, for $n \in \mathbb{Z}_+$, $\rho_n(x) = n \rho (nx)$. Let $H_n (x) = \rho_n * H (x - \frac{1}{n})$. $H_n$ is smooth by the properties of mollifiers, and it is easy to see that $0 \leq H_n(x) \leq 1$, $H_n(x) = 1$ for $x \geq 2/n$ and $H_n(x) = 0$ for all $x \leq 0$. 

Consider the following equations which approximate \eqref{eq1.1}-\eqref{eq1.3}: 
\begin{equation} \label{eq2.1}
\partial_t \phi^n + v_n \cdot \nabla \phi^n = 0,
\end{equation}
\begin{equation} \label{eq2.2}
\phi^n(x,0) = \phi_0(x),
\end{equation}
\begin{equation} \label{eq2.3}
v_n(x,t) = \omega_0 \int_{\mathbb{R}^2} K(x - y) H_n(\phi^n(y,t)) dy.
\end{equation}

\begin{definition}
We say that $\phi^n:\mathbb{R}^2 \times \mathbb{R} \rightarrow \mathbb{R}$ is a solution to the $n$-approximate vortex patch equations with $C^{k, \, \mu}$ initial scalar ($n \in \mathbb{Z}_+$) if it solves equations \eqref{eq2.1}-\eqref{eq2.3} and $\phi_0$ is a $C^{k, \, \mu}$ initial scalar. 
\end{definition}

Define the initial vorticity
\begin{equation} \label{eq2.4}
\omega_0^n(x) = \omega_0 H_n(\phi_0(x)).
\end{equation}
Since $H_n$ is smooth and $\phi_0$ is $C^{k, \, \mu}$, $\omega_0^n \in C^{k, \, \mu}(\mathbb{R}^2)$. Moreover, $\omega_0^n$ is supported in $\overline{D_0}$ (the closure of $D_0$), so it is compactly supported. Therefore, there exists for all time a unique solution to the 2D vorticity equations (see chapter 4 of \cite{MB}):
\begin{equation} \label{eq2.5}
\partial_t \omega^n + v_n \cdot \nabla \omega^n = 0
\end{equation}
\begin{equation} \label{eq2.6}
\omega^n(x,0) = \omega_0^n.
\end{equation}
\begin{equation} \label{eq2.7}
v_n(x,t) = \int_{\mathbb{R}^2} K(x - y) \omega^n(y,t) dy. 
\end{equation}
The velocity field $v_n$ defines the particle-trajectory mappings $X_n: \mathbb{R}^2 \times \mathbb{R} \rightarrow \mathbb{R}^2$ via the equation 
\begin{equation} \label{eq2.8}
\frac{d}{dt} X_n (\alpha, t) = v_n (X_n(\alpha,t), t), \, \, \, \, \, X_n(\alpha, 0) = \alpha.
\end{equation}
For each time $t$, $X_n(\cdot, t)$ is a volume-preserving $C^{k+1}$ diffeomorphism of the plane onto itself, with inverse $X_n^{-1}(\cdot, t)$. We will refer to $X_n^{-1}(x,t)$ as the back-to-labels mappings. In two dimensions, the vorticity is transported by the particle-trajectories. That is,
\begin{equation} \label{eq2.9}
\omega^n(x,t) = \omega_0^n (X_n^{-1} (x,t)) = \omega_0 H_n \circ \phi_0(X_n^{-1}(x,t)). 
\end{equation}

Define $\phi^n (x, t) = \phi_0 (X_n^{-1}(x,t))$. We show that this is a solution to the $n$-approximate equations. 

\begin{proposition} \label{prop2.1}
Assume $\phi_0$ is a $C^{k, \, \mu}$ initial scalar. Let $\omega_0^n$ be an initial vorticity defined as in \eqref{eq2.4}, and let $X_n^{-1}(x,t)$ be the back-to-labels mappings corresponding to the solution $(\omega^n, v_n)$ of the vorticity equations \eqref{eq2.5}-\eqref{eq2.7}. Then $\phi^n(x,t) = \phi_0(X_n^{-1}(x,t))$ is a solution to the $n$-approximate equations with $C^{k, \, \mu}$ initial scalar. Moreover, $\phi^n(\cdot, t) \in C^{k, \, \mu}(\mathbb{R}^2)$ for all times $t \in \mathbb{R}$  and $\nabla v_n \in L^1_{loc}(\mathbb{R}; L^\infty(\mathbb{R}^2))$.
\end{proposition}

\begin{proof}
Changing to Lagrangian coordinates, we see that $\phi^n(X_n(\alpha, t), t) = \phi_0 (\alpha)$. Differentiating in time and using \eqref{eq2.8}, we obtain: 
\begin{eqnarray*}
	0 & = & \frac{d}{dt} X_n(\alpha, t) \cdot \nabla \phi^n(X_n(\alpha,t), 	t) + \partial_t \phi^n(X_n(\alpha, t), t) \\
	   & = & v_n(X_n(\alpha, t), t) \cdot \nabla  \phi^n(X_n(\alpha,t), t) + \partial_t \phi^n(X_n(\alpha, t), t). 
\end{eqnarray*}
Since $X_n(\cdot, t)$ is a bijection, this shows that \eqref{eq2.1} is satisfied, provided that $v_n$ satisfies \eqref{eq2.3}. The initial condition in \eqref{eq2.8} shows \eqref{eq2.2} is satisfied. To see that the velocity $v_n$ satisfies \eqref{eq2.3}, note that \eqref{eq2.9} together with the definition of $\phi^n$ yield $\omega^n(x,t) = \omega_0 H_n (\phi^n(x,t))$, and the result follows by plugging this into \eqref{eq2.7}. 

It remains to show that $\phi^n(\cdot, t) \in C^{k, \, \mu}$ and that $\nabla v_n \in L^1_{loc}(\mathbb{R}; L^\infty(\mathbb{R}^2))$. It is well known that if $\omega_0^n \in C^{k, \, \mu}$ is compactly supported, then the velocity $v_n$ satisfies this property and $\nabla X_n(\cdot, t) \in C^{k, \, \mu}$ (see chapter 4 in \cite{MB}). Furthermore, by writing $X_n^{-1}(X_n(\alpha, t), t) = \alpha$ we find 
\begin{equation*}
\nabla_x X_n^{-1}(x, t) = \frac{\Co[\nabla_\alpha X_n(\alpha, t)]}{\det \nabla_\alpha X_n(\alpha, t)} = \Co[\nabla_\alpha X_n(\alpha, t)],
\end{equation*}
where $\Co[\nabla_\alpha X_n]$ is the cofactor matrix of $\nabla_\alpha X_n$. Therefore, we also have $\nabla X_n^{-1}(\cdot ,t) \in C^{k, \, \mu}$.  By construction, $| \phi^n(\cdot, t)|_{L^\infty} = |\phi_0|_{L^\infty}$. Also, 
\begin{equation*}
\nabla \phi^n(x, t) = (\nabla X_n^{-1}(x,t))^T \nabla \phi_0(X_n^{-1}(x,t)).
\end{equation*} 
This last expression readily implies $\nabla \phi^n(\cdot, t) \in C^{k-1, \, \mu}(\mathbb{R}^2)$ and concludes the proof. 
\end{proof}

In the following proposition, we note that the velocities of these solutions are bounded and that the area of the $n$-approximate patch is constant. 

Let $D_n = D_n(t) = \{x \in \mathbb{R}^2 \mid \phi^n(x,t) > 0\}$ (clearly, $D_n(0) = D_0$) and define $\pi L^2 =  m(D_0)$, where $m(\cdot)$ is the Lebesgue measure on $\mathbb{R}^2$. We will denote the ball of radius $L$ around the origin by $B_L(0)$.

\begin{proposition} \label{prop2.2}
Let $\phi^n$ be the solution given by proposition \ref{prop2.1}, and let $v_n$ be the corresponding velocity. Then, for any time $t \in \mathbb{R}$, $m(D_n) = m(D_0)$ and
\begin{equation*}
|v_n(\cdot, t)|_{L^\infty} \leq |\omega_0| \sqrt{\frac{m(D_0)}{\pi}}.
\end{equation*}
\end{proposition}

\begin{proof}
The first claim follows immediately from the facts that $\phi^n$ is transported by the particle-trajectories $\phi^n(x,t) = \phi_0(X_n^{-1}(x,t))$ and that these are volume-preserving. Therefore, $m(B_L(0)) = m(D_n)$. For the second, we have 
\begin{eqnarray*}
|v_n(x,t)| & \leq & |\omega_0| \int_{\mathbb{R}^2} \frac{1}{2\pi} \frac{1}{|x-y|} H_n(\phi^n(y,t)) dy \\ 
& \leq & \frac{|\omega_0|}{2\pi} \int_{D_n} \frac{1}{|x - y|} dy,
\end{eqnarray*} 
where for the last inequality we used that $0 \leq H_n \leq 1$ and that $H_n(\phi^n(x,t)) = 0$ whenever $\phi^n(x,t) \leq 0$. Since the function $\frac{1}{|x|}$ is radially decreasing, we have 
\begin{equation*}
\frac{1}{2\pi} \int_{D_n} \frac{1}{|x - y|} dy  \leq \frac{1}{2\pi} \int_{B_L(0)} \frac{1}{|y|} dy = L. 
\end{equation*}
The second claim follows. 
\end{proof}

We conclude this section by showing that the boundary $\partial D_n$ remains $C^{k, \, \mu}$ for all time. By the inverse mapping theorem, we only need $|\nabla \phi^n (\cdot, t)|_{\inf} > 0$. However, we will prove a slightly stronger result which will be of use in the next section. Let us introduce the notation
\begin{equation*}
|\nabla \phi|_{M-\inf} = \inf \{ |\nabla \phi(x)| \mid x \in \phi^{-1} ([0, 2/M] \}.  
\end{equation*} 
Note that the set $(\phi^n(\cdot, t))^{-1} ([0, 2/n])$ is the domain on which $\omega^n(\cdot, t) = \omega_0 H_n(\phi^n(\cdot,t))$ differs from the step function $\omega_0H(\phi^n(\cdot,t))$. 

\begin{lemma} \label{lemma2.1}
Let $\phi_0$ be a $C^{k, \, \mu}$ initial scalar with $|\nabla \phi_0|_{\inf} \geq m > 0$. Then for all sufficiently large $n \in \mathbb{Z}_+$, $|\nabla \phi_0|_{n - \inf} \geq m/2$.  
\end{lemma}

\begin{proof}
Suppose the claim is false. Then there exists a sequence $x_n \in \phi_0^{-1} ([0, 2/n])$ such that $|\nabla \phi_0 (x_n)| < m/2$. Since $\overline{D_0}$ is compact and $\phi_0^{-1} ([0, 2]) \subset \overline{D_0}$ is closed it follows that the sequence is a subset of the compact set $\phi_0^{-1} ([0, 2])$. Therefore, it has a converging subsequence, which we still denote by $x_n$. Let $x$ be the limit. Since for all $n > N$, $x_n \in \phi_0^{-1}([0,2/N])$ which is closed, it follows that
\begin{equation*}
x \in \bigcap_{n = 1}^{\infty} \phi_0^{-1} ([0, 2/n]) = \phi_0^{-1}(\{0\}). 
\end{equation*}
But since $\nabla \phi_0$ is continuous and $|\nabla \phi_0(x_n)| < m/2$ for all $n$, it follows that $|\nabla \phi_0(x)| \leq m/2 < m$, which is a contradiction. 
\end{proof}

We will use the notation $\nabla^\perp = [-\partial_2 \, \, \, \, \partial_1]^T$. 

\begin{proposition} \label{prop2.3}
Assume that $\phi_0$ is a $C^{k, \, \mu}$ initial scalar. Let $(\phi^n, v_n)$ be a solution given by proposition \ref{prop2.1}, and let $M \in \mathbb{Z}_+$ sufficiently large such that $|\nabla \phi_0|_{M-\inf} \geq m/2$.  For all times $t \in \mathbb{R}$, we have 
\begin{equation} \label{eq2.10}
\frac{d}{dt} \nabla^\perp \phi^n(X_n(\alpha, t), t) = \nabla v_n \nabla^\perp \phi^n (X_n(\alpha, t), t).
\end{equation}
Consequently, 
\begin{equation*}
|\nabla \phi^n(\cdot, t)|_{M-\inf} \geq |\nabla \phi_0|_{M-\inf}\exp \bigg[ - \int_0^t |\nabla v_n(\cdot, s)|_{L^\infty} ds \bigg],
\end{equation*}
\begin{equation*}
|\nabla \phi^n(\cdot, t)|_{L^\infty} \leq |\nabla \phi_0|_{L^\infty} \exp \bigg[ \int_0^t |\nabla v_n(\cdot, s)|_{L^\infty} ds \bigg] .
\end{equation*}
\end{proposition}

\begin{proof}
By differentiating the expression $\phi^n(X_n(\alpha, t), t) = \phi_0(\alpha)$ with respect to the Lagrangian space variables we obtain 
\begin{eqnarray*}
\nabla^\perp \phi_0(\alpha) = \begin{bmatrix}
\partial_{\alpha_2}X_n^2 & - \partial_{\alpha_2} X_n^1 \\ 
- \partial_{\alpha_1} X_n^2 & \partial_{\alpha_1}X_n^1
\end{bmatrix} (\alpha, t) \nabla^\perp \phi^n (X_n(\alpha, t), t).   
\end{eqnarray*}
Note that the matrix above is the inverse of $\nabla_\alpha X_n(\alpha, t)$. Therefore, if we denote $W_n = \nabla^\perp \phi^n$, $W_0 = \nabla^\perp \phi_0$:
\begin{equation*}
W_n(X_n(\alpha, t),t) = \nabla_\alpha X_n(\alpha, t)W_0(\alpha).
\end{equation*}
From \eqref{eq2.8}, we have
\begin{equation*}
\frac{d}{dt} \nabla_\alpha X_n(\alpha, t) = \nabla v_n (X_n(\alpha, t), t) \nabla_\alpha X_n(\alpha, t),
\end{equation*} 
and, therefore
\begin{equation*}
\frac{d}{dt}W_n(X_n(\alpha, t), t) = \nabla v_n(X_n(\alpha, t), t) \nabla_\alpha X_n(\alpha,t) W_0(\alpha) = \nabla v_n W_n(X_n(\alpha, t), t), 
\end{equation*}
as asserted. 

If we now denote $Z_n(\alpha, t) = W_n(X_n(\alpha, t), t)$,
\begin{equation*}
\frac{d}{dt} Z_n(\alpha, t) = \nabla v_n(X_n(\alpha, t), t) Z_n(\alpha, t),
\end{equation*}
so 
\begin{equation*}
\frac{d}{dt} \log |Z_n(\alpha, t)| \leq |\nabla v_n (X_n(\alpha, t),t)|,
\end{equation*}
which implies that 
\begin{equation*}
\exp \bigg[ - \int_0^t |\nabla v_n(\cdot, s)|_{L^\infty} ds \bigg] \leq \frac{|Z_n(\alpha, t)|}{|Z_n(\alpha, 0)|} \leq \exp \bigg[ \int_0^t |\nabla v_n(\cdot, s)|_{L^\infty} ds \bigg].
\end{equation*}
The conclusions follow.
\end{proof}

\begin{remark}
If $\phi_0 \in C^{k, \, \mu}$ with $k \geq 2$ one can obtain \eqref{eq2.10} in a slightly more direct way: it follows from \eqref{eq2.1} and the divergence-free property of the velocity $v_n$ that 
\begin{equation*}
\partial_t W_n + v_n \cdot \nabla W_n = \nabla v_n W_n.
\end{equation*}
Differentiating $W_n(X_n(\alpha, t), t)$ with respect to time immediately gives \eqref{eq2.10}. 
\end{remark}

\section{Uniform \texorpdfstring{$C^{1,\, \mu}$}{C(1,mu)}  Estimates for Solutions of the Approximate Equations}

In this sections we prove the bounds for the solutions of the approximate equations which will allow us to show convergence. The arguments are in essence those applied in \cite{BC} to equations \eqref{eq1.1}-\eqref{eq1.3}. The main difference between the arguments in this section and those presented there is that proving the bounds for $|\nabla v_n|_{L^\infty}$ is slightly more involved due to the more complicated expressions for the vorticities. The idea is to approximate the vorticities $\omega^n$ by a sum of step functions. Each such step function, treated as a vorticity, gives rise to a velocity via the Biot-Savart law, and we can apply the considerations of \cite{BC} in order to bound its gradient. A convergence argument then gives us what we wanted. However, for this to work, we require $|\nabla \phi|_{M-\inf} > 0$, rather than just $|\nabla \phi|_{\inf} > 0$. Lemma \ref{lemma2.1} shows that for large enough $M$, no further further assumptions on $\phi_0$ are necessary.

The following geometric lemma, which is proved in \cite{BC}, states roughly that the intersection of the domain $D = \{x \in \mathbb{R}^2 \mid \phi(x) > 0\}$ (defined by a $C^{1, \, \mu}(\mathbb{R}^2)$ function) with small circles centered close to the boundary is approximately a semi-circle. It is the essential result in the proof of the bounds. 

Let $\phi \in C^{1, \, \mu}(\mathbb{R}^2)$ such that $D = \{x \in \mathbb{R}^2 \mid \phi(x) > 0\}$ is bounded, and $|\nabla \phi|_{\inf} > 0$. Let $d(x_0) = \inf_{x \in \partial D} |x - x_0|$. Define 
\begin{equation} \label{eq3.1}
\delta = \bigg( \frac{|\nabla \phi|_{\inf}}{|\nabla \phi|_\mu} \bigg)^{\frac{1}{\mu}}.
\end{equation}
For $\rho \geq d(x_0)$ consider the set of directions 
\begin{equation*}
S_\rho (x_0) = \{z \in \mathbb{S}^1 \mid x_0 + \rho z \in D \},
\end{equation*}
where $\mathbb{S}^1$ is the unit circle in $\mathbb{R}^2$. There exists $\tilde{x} \in \partial D$, such that $|\tilde{x} - x_0| = d(x_0)$. Consider the semi-circle 
\begin{equation*}
\Sigma(x_0) = \{ z \in S^1 \mid \nabla \phi(\tilde{x}) \cdot z \geq 0 \}.
\end{equation*}
Finally, we define the symmetric difference 
\begin{equation} \label{eq3.2}
R_\rho (x_0) = (S_\rho(x_0) \setminus \Sigma(x_0)) \cup (\Sigma(x_0) \setminus S_\rho(x_0)).
\end{equation}

\begin{lemma} (Geometric Lemma) \label{lemma3.1}
Let $R_\rho(x_0)$ be the symmetric difference defined in \eqref{eq3.2} and let $\mathcal{H}^1$ denote the Lebesgue measure on the unit circle. Then, 
\begin{equation*}
\mathcal{H}^1(R_\rho(x_0)) \leq 2 \pi \bigg( (1 + 2^\mu)\frac{d(x_0)}{\rho} + 2^\mu \bigg( \frac{\rho}{\delta} \bigg)^\mu \bigg)
\end{equation*}
holds for all $\rho \geq d(x_0)$, $\mu \in (0,1)$ and $x_0 \in \mathbb{R}^2$ such that $d(x_0) < \delta$, where $\delta$ is given by \eqref{eq3.1}.
\end{lemma}

As discussed, in order to apply the geometric lemma to bound the $L^\infty$ norm of $\nabla v_n$, we will need to approximate $\omega^n$ by a sum of step functions. To do this, we have the following elementary lemma: 

\begin{lemma} \label{lemma3.2}
Let $H_n: \mathbb{R} \rightarrow \mathbb{R}$ be defined as in section 2. For each $N \in \mathbb{Z}_+$, consider 
\begin{equation} \label{eq3.3}
H_n^{N} (x) = \sum_{k = 1}^N \bigg[ H_n \bigg( \frac{2k}{nN} \bigg) - H_n \bigg( \frac{2(k-1)}{nN} \bigg) \bigg] H \bigg(x - \frac{2k}{nN}\bigg). 
\end{equation}
Then, $H_n^N$ converge uniformly to $H_n$ as $N \rightarrow \infty$.
\end{lemma}

\begin{proof}
For $x \geq 2/n$, $H(x - 2k/nN) = 1$ for all $k$, and we obtain after cancelations $H_n^N(x) = H_n(2/n) - H_n(0) = 1$. Therefore $H_n^N(x) = H_n (x)$ for all $N$. For $x \leq 0$, $H(x - 2k/nN) = 0$ for all $k$, so $H_n^N(x) = H_n(x) = 0$ for all $N$. If $x = 2k/nN$ for some $N \geq k \geq 1$, we obtain after cancellations that $H_n^N (x) = H_n(2k/nN) - H_n(0) = H_n(2k/nN)$. Finally, if $x \in (2(k-1)/nN, 2k/nN)$, then $H_n^N (x) = H_n^N(2(k-1)/nN) = H_n(2(k-1)/nN)$. Therefore, by the mean value theorem, 
\begin{equation*}
|H_n(x) - H_n^N(x)| \leq |H_n'|_{L^\infty} |x - 2(k-1)/nN| \leq 2n/nN = 2/N,
\end{equation*}
where we used the fact that $|H_n'|_{L^\infty} \leq n$, which is easy to see from the definition of $H_n$ and the basic properties of mollifiers. Thus, we have obtained that 
\begin{equation*}
|H_n - H_n^N|_{L^\infty} \leq \frac{2}{N},
\end{equation*}
and we conclude that $H_n^N$ converge uniformly to $H_n$ as $N \rightarrow \infty$. 
\end{proof}

Before we go further, let us recall that the gradient of the velocity obeys the equation
\begin{equation} \label{eq3.4}
\nabla v_n (x, t) = \frac{1}{2\pi} \PV \int_{\mathbb{R}^2} \frac{\sigma(x-y)}{|x-y|^2}  \omega^n(y,t) dy + \frac{\omega^n(x,t)}{2} \begin{bmatrix}
0 & -1 \\
1 & 0
\end{bmatrix},
\end{equation}
where PV denotes the Cauchy principle-value, and 
\begin{equation*}
\sigma(x) = \frac{1}{|x|^2} \begin{bmatrix}
2 x_1 x_2 & x_2^2 - x_1^2 \\
x_2^2 - x_1^2 & - 2 x_1 x_2
\end{bmatrix}.
\end{equation*}
The important characteristics of $\sigma(x)$ are that it is smooth away from the origin, it is homogeneous of degree 0, it has mean zero on the unit circle $\int_{\mathbb{S}^1} \sigma(x) dS(x) = 0$, and it is symmetric with respect to reflections $\sigma (x) = \sigma(-x)$. 

Fix $M \in \mathbb{Z}_+$ sufficiently large such that for all $n \geq M$ the claim of lemma \ref{lemma2.1} holds. From now on, we restrict the discussion to such large $n$. In view of \eqref{eq3.3}, consider for some fixed $r \in [0,2/n] \subset [0,2/M]$,
\begin{equation} \label{eq3.5}
w_r(x,t) = \frac{\omega_0}{2\pi} \PV \int_{\mathbb{R}^2 }\frac{\sigma (x - y)}{|x-y|^2} H(\phi^n(y,t) - r) dy.
\end{equation}
Let $D_n^r = \{x \in \mathbb{R}^2 \mid \phi^n(x,t) > r\}$, $\pi L_r^2 = m(D_n^r)$, $d_r(x) = \inf_{y \in \partial D_n^r} |x - y|$ and 
\begin{equation*}
\delta_r = \bigg( \frac{\inf \{|\nabla \phi^n(x,t)| \mid x \in (\phi^n(\cdot, t))^{-1} (\{r\})\}}{|\nabla \phi^n(\cdot, t)|_\mu} \bigg)^{\frac{1}{\mu}} \geq  \bigg( \frac{|\nabla \phi^n(\cdot, t)|_{M - \inf}}{|\nabla \phi^n(\cdot, t)|_\mu} \bigg)^{\frac{1}{\mu}} > 0. 
\end{equation*}
We now show that the $L^\infty$ norm of $w_r$ is bounded by an expression of $\delta_r$ and $L_r$. The following proposition is proved in \cite{BC} with slightly different notation. For the convenience of the reader, we repeat the proof here. 

\begin{proposition} \label{prop3.1}
Let $w_r$ be given by \eqref{eq3.5} and let $\delta_r$ and $L_r$ as above. Suppose $M$ is sufficiently large such that lemma \ref{lemma2.1} holds and let $n \geq M$. Then, 
\begin{equation} \label{eq3.6}
|w_r(\cdot, t)|_{L^\infty} \leq C(\mu) |\omega_0| \bigg( 1 + \log \bigg(1 + \frac{L_r}{\delta_r} \bigg) \bigg),
\end{equation}
where the constant $C(\mu)$ depends only on the H\"older coefficient $\mu$.
\end{proposition}

\begin{proof}
We split the integral,
\begin{equation*}
w_r(x,t) = \frac{\omega_0}{2 \pi} \bigg[ \int_{D_n^r \cap \{|x - y| \geq \delta_r\}} \frac{\sigma (x - y)}{|x-y|^2}  dy + \PV \int_{D_n^r \cap \{|x - y| \leq \delta_r\}}  \frac{\sigma (x - y)}{|x-y|^2} dy \bigg] = I_1(x) + I_2(x). 
\end{equation*}

For $I_1(x)$, since $\frac{1}{|x|^2}$ is radially decreasing and $m(D_n^r \cap \{|x - y| \geq \delta_r\}) \leq m(D_n^r) = \pi L_r^2 \leq \pi (L_r+ \delta_r)^2 - \pi \delta_r^2$, we have after changing to polar coordinates:
\begin{equation*}
|I_1(x)| \leq \sqrt{2} |\omega_0| \int_{\delta_r}^{L_r + \delta_r} \frac{1}{r} dr = \sqrt{2}|\omega_0| \log \bigg(1 + \frac{L_r}{\delta_r} \bigg),  
\end{equation*}
where we used $|\sigma(x)| = \sqrt{2}$. 

For $I_2(x)$, we distinguish two cases. If $d_r(x) \geq \delta_r$, then either $D_n^r \cap \{|x - y| \leq \delta_r\} = \emptyset$ if $x$ is in the complement of $D_n^r$, so $I_2(x) = 0$; or $D_n^r \cap \{|x - y| \leq \delta_r\} = \{|x - y| \leq \delta_r\}$ if $x \in D_n^r$, in which case we have $I_2(x) = 0$ since $\sigma$ has mean zero on circles. If on the other hand, $d_r(x) < \delta_r$, we change to polar coordinates and use the fact that the integral of $\sigma$ is zero on semi-circles (so in particular on $\Sigma(x)$) to obtain 
\begin{equation*}
|I_2(x)| \leq \sqrt{2} \frac{|\omega_0|}{2 \pi} \int_{d_r(x)}^{\delta_r} \mathcal{H}^1(R_\rho(x))\frac{1}{\rho} d\rho. 
\end{equation*}
Applying lemma \ref{lemma3.1} and integrating, we find 
\begin{equation*}
|I_2(x)| \leq \sqrt{2} |\omega_0| \bigg[ (1 + 2^\mu) + \frac{2^\mu}{\mu} \bigg].
\end{equation*}
The conclusion follows. 
\end{proof}

\begin{remark}
For $\epsilon > 0$, let
\begin{equation*}
w_r^\epsilon(x,t) = \frac{\omega_0}{2\pi} \int_{\{|x-y| \geq \epsilon\} }\frac{\sigma (x - y)}{|x-y|^2} H(\phi^n(y,t) - r) dy.
\end{equation*}
The $L^\infty$ norm of this quantity is also bounded by the right-hand side of \eqref{eq3.6} if we choose $\epsilon < \delta_r$. Indeed, the proof is exactly as above, with the only possibly contentious case presented by $d_r(x) < \epsilon < \delta_r$, but in this case we simply have  
\begin{equation*}
|I_2(x)| \leq \sqrt{2} \frac{|\omega_0|}{2 \pi} \int_{\epsilon}^{\delta_r} \mathcal{H}^1(R_\rho(x))\frac{1}{\rho} d\rho \leq \sqrt{2} \frac{|\omega_0|}{2 \pi} \int_{d_r(x)}^{\delta_r} \mathcal{H}^1(R_\rho(x))\frac{1}{\rho} d\rho.
\end{equation*}
\end{remark}

Let us now define the following approximation of the symmetric part of $\nabla v_n$: 
\begin{equation} \label{eq3.7}
w^N(x,t) = \frac{\omega_0}{2\pi} \PV \int_{\mathbb{R}^2} \frac{\sigma(x-y)}{|x-y|^2} H_n^N(\phi^n(y,t))dy,
\end{equation}
Where $H_n^N$ is defined by \eqref{eq3.3}. Let $\pi L^2 = m(D_n) = m(D_0)$ and define 
\begin{equation*}
\Delta_n = \bigg( \frac{|\nabla \phi^n(\cdot, t)|_{M - \inf}}{|\nabla \phi^n(\cdot, t)|_\mu} \bigg)^{\frac{1}{\mu}}.
\end{equation*}

\begin{corollary} \label{cor3.1}
Let $w^N$ be given by \eqref{eq3.7} and let $L$ and $\Delta_n$ as above. Suppose $M$ is sufficiently large such that lemma \ref{lemma2.1} holds and let $n \geq M$. Then, 
\begin{equation} \label{eq3.8}
|w^N(\cdot, t)|_{L^\infty} \leq C(\mu) |\omega_0| \bigg( 1 + \log \bigg(1 + \frac{L}{\Delta_n} \bigg) \bigg),
\end{equation} 
where $C(\mu)$ is the constant given by proposition \ref{prop3.1}. 
\end{corollary}

\begin{proof}
First, let us note that for each $r \in [0, 2/n]$, $D_n^r \subset D_n$, so $L_r \leq L$. Moreover, it is also clear by definition that $\Delta_n \leq \delta_r$. Therefore, plugging this into \eqref{eq3.6} we obtain that for all $r \in [0,2/n]$, 
\begin{equation*}
|w_r(\cdot, t)|_{L^\infty}  \leq C(\mu) |\omega_0| \bigg( 1 + \log \bigg(1 + \frac{L}{\Delta_n} \bigg) \bigg). 
\end{equation*}

On the other hand, 
\begin{equation*}
w^N(x,t) = \sum_{k=1}^N \bigg[ H_n \bigg( \frac{2k}{nN} \bigg) - H_n \bigg( \frac{2(k-1)}{nN} \bigg) \bigg] \frac{\omega_0}{2\pi} \PV \int_{\mathbb{R}^2} \frac{\sigma (x - y)}{|x-y|^2} H(\phi^n(y,t) - 2k/nN) dy.
\end{equation*}
Each term (ignoring the constant in the parenthesis) is of the form $w_r$, with $r = 2k/nN$. Therefore, if we apply the triangle inequality and note the fact that $H_n$ is monotonically increasing (so the constant in parenthesis is non-negative for all $k$) we obtain, after cancellations, the wanted inequality. 
\end{proof}

\begin{remark}
For $\epsilon > 0$, we define 
\begin{equation*}
w^N_\epsilon(x,t) = \frac{\omega_0}{2\pi} \int_{\{|x-y| \geq \epsilon\}} \frac{\sigma(x-y)}{|x-y|^2} H_n^N(\phi^n(y,t))dy.
\end{equation*}
If we choose $\epsilon < \Delta_n \leq \delta_r$ (any $r \in [0, 2/n] \subset [0, 2/M]$), $|w^N_\epsilon(\cdot, t)|_{L^\infty}$ is also bounded by the right-hand side of \eqref{eq3.8}. See the remark after proposition \ref{prop3.1}. 
\end{remark}

We are now ready to prove the desired bound for the $L^\infty$ norm of $\nabla v_n$. 

\begin{proposition} \label{prop3.2}
Let $(\phi^n, v_n)$ be a solution given by proposition \ref{prop2.1}, for $n \geq M$ and $M$ sufficiently large such that lemma \ref{lemma2.1} holds. Then, 
\begin{equation*}
|\nabla v_n(\cdot, t)|_{L^\infty} \leq C |\omega_0| \bigg( 1 + \log \bigg(1 + \frac{L}{\Delta_n} \bigg) \bigg),
\end{equation*}
where the constant $C$ only depends on $\mu$ (in particular, it is independent of time, or $n$). 
\end{proposition}

\begin{proof}
Clearly, the second term in \eqref{eq3.4} satisfies this bound since $|\omega^n|$ is bounded by $|\omega_0|$. So, we only have to address the symmetric part of $\nabla v_n$. Let $\epsilon > 0$. We have that for all $N \in \mathbb{Z}_+$
\begin{eqnarray*}
\bigg| \frac{1}{2\pi} \int_{|x - y| \geq \epsilon} \frac{\sigma(x-y)}{|x-y|^2}  \omega^n(y,t) dy \bigg| & \leq & \bigg| w_\epsilon^N(x,t) \bigg| + \\
& + & \bigg| \frac{\omega_0}{2\pi} \int_{|x-y| \geq \epsilon} \frac{\sigma(x-y)}{|x-y|^2} \bigg[H_n(\phi^n(y,t)) - H_n^N(\phi^n(y,t)) \bigg]dy \bigg|,
\end{eqnarray*}
where $w_\epsilon^N$ is defined as in the previous remark. We choose $\epsilon < \Delta_n$ (see the remark after corollary \ref{cor3.1}), to obtain 
\begin{equation*}
| w_\epsilon^N(x,t)| \leq C(\mu) |\omega_0| \bigg( 1 + \log \bigg(1 + \frac{L}{\Delta_n} \bigg) \bigg). 
\end{equation*}
On the other hand, as seen in the proof of lemma \ref{lemma3.2}, $H_n(\phi^n(y,t)) - H_n^N(\phi^n(y,t))$ is nonzero only for $y \in (\phi^n(\cdot, t))^{-1}([0, 2/n]) \subset \overline{D_n}$. Therefore, 
\begin{equation*}
\bigg| \frac{\omega_0}{2\pi} \int_{|x-y| \geq \epsilon} \frac{\sigma(x-y)}{|x-y|^2} \bigg[H_n(\phi^n(y,t)) - H_n^N(\phi^n(y,t)) \bigg]dy \bigg| \leq \frac{\sqrt{2}|\omega_0|}{2\pi} \frac{1}{\epsilon^2} m(D_0) |H_n(\cdot) - H_n^N(\cdot)|_{L^\infty}.
\end{equation*}
Since this holds for all $N$, by lemma \ref{lemma3.2}, we can choose $N = N(\epsilon)$ sufficiently large such that 
\begin{equation*}
\bigg| \frac{1}{2\pi} \int_{\{|x - y| \geq \epsilon\}} \frac{\sigma(x-y)}{|x-y|^2}  \omega^n(y,t) dy \bigg| \leq C(\mu) |\omega_0| \bigg( 1 + \log \bigg(1 + \frac{L}{\Delta_n} \bigg) \bigg) + \epsilon,
\end{equation*}
for all $\epsilon < \Delta_n$. Taking the limit $\epsilon \rightarrow 0$ concludes the proof. 
\end{proof} 

We now quote a general lemma that is proved in \cite{BC}, and also add a simple observation of which we will make use in section 6. 

\begin{lemma} \label{lemma3.3}
Let K be a Calderon-Zygmund kernel, homogeneous of degree $-N$, with mean zero on spheres, satisfying $|\nabla K(x)|\leq C |x|^{-N-1}$. There exists a constant $C_0$, so that all $f \in C^{0, \, \mu}(\mathbb{R}^N)$ and $\omega \in L^\infty(\mathbb{R}^N)$ satisfy
\begin{equation} \label{eq3.9}
|G|_\mu \leq C_0 (\mu, N)|f|_\mu (|K*\omega|_{L^\infty} + |\omega|_{L^\infty}),
\end{equation}
where 
\begin{equation*}
G(x) = \PV \int_{\mathbb{R}^N} K(x-y)(f(x) - f(y)) \omega(y) dy.
\end{equation*}
If, moreover, the support of $\omega$ has finite Lebesgue measure, let $m(\supp(\omega)) = m(B_1) L^N$, where $B_1$ is the unit ball in $\mathbb{R}^N$. Then, for each $R>0$,  
\begin{equation} \label{eq3.10}
|G|_{L^\infty} \leq C_0 |\omega|_{L^\infty}\bigg(|f|_\mu R^\mu + |f|_{L^\infty} \log \big(1 + \frac{L}{R} \big)\bigg).
\end{equation}
\end{lemma}

\begin{proof}
The first claim, \eqref{eq3.9}, is proved in \cite{BC}. We only prove \eqref{eq3.10}. Let $R>0$. 
\begin{eqnarray*}
|G(x)| & \leq &  \int_{|x-y| < R}|f(x) - f(y)||K(x-y)||\omega(y)| dy + \\ 
& & + \int_{|x-y| \geq R}|f(x) - f(y)||K(x-y)||\omega(y)| dy \\ 
& = & T_1 + T_2. 
\end{eqnarray*}
For the first term we use the H\"older semi-norms: 
\begin{equation*}
T_1 \leq C |f|_{\mu} |\omega|_{L^\infty}\int_{|y| \leq R} |y|^{\mu - N} dy = C|f|_\mu |\omega|_{L^\infty} R^\mu.
\end{equation*}
For the second, 
\begin{equation*}
T_2 \leq C |f|_{L^\infty}|\omega|_{L^\infty} \int_{\supp(\omega) \cap \{|x-y| \geq R\}} \frac{1}{|x-y|^N}dy.
\end{equation*}
By the radial monotonicity of $\frac{1}{|x|^N}$ and the fact that $m(\supp(\omega) \cap \{|x-y| \geq R\}) \leq m(\supp(\omega)) = m(B_1) L^N \leq m(B_1)[(L+R)^N - R^N] = m(B_{L+R} \setminus B_{R}) $, we can bound the integral above: 
\begin{equation*}
\int_{\supp(\omega) \cap \{|x-y| \geq R\}} \frac{1}{|x-y|^N}dy \leq C \int_R^{L+R} \frac{1}{r} dr = C \log \big(1 + \frac{L}{R}\big).
\end{equation*}
The conclusion follows. 
\end{proof}

\begin{remark}
In fact, we have shown that if $\omega$ has support of finite measure (which is, indeed, the case for our sequence of interest $\omega^n$), then the integrand in the definition of $G$ is in $L^1$, and, therefore, by the dominated convergence theorem, the principal-value converges to the integral and can be discarded.
\end{remark}

The proof of the proposition that follows is very similar to that given in \cite{BC}, but here, the arguments have to be slightly modified in order to treat the non-constant vorticity.

\begin{proposition} \label{prop3.3}
Let $(\phi^n, v_n)$ be a solution given by proposition \ref{prop2.1}, and let $W_n = \nabla^\perp \phi_n$. Then, 
\begin{equation} \label{eq3}
\nabla v_n(x) W_n(x) = \frac{1}{2\pi} \PV \int_{\mathbb{R}^2} \frac{\sigma(x-y)}{|x-y|^2}\omega^n(y)(W_n(x) - W_n(y)) dy. 
\end{equation}
\end{proposition}

\begin{proof}
The result relies on applying Green's theorem and noticing that all but the boundary terms cancel due to the the fact that $W_n \omega^n$ is (weakly) divergence-free. However, $W_n$ need not be differentiable. To rectify this, let $\rho \in C_0^\infty(\mathbb{R}^2)$ be a standard mollifier: $0 \leq \rho(x) \leq 1$, $\rho(x) = 0$ for $|x| \geq 1$ and $\int_{\mathbb{R}^2} \rho(x) dx =1$. For $\delta > 0$, denote $\rho_\delta(x) = \delta^{-2} \rho(x/\delta)$, and $W_n^\delta = \rho_\delta * W_n$. Then, note that
\begin{equation*}
W_n^\delta(x) = \int_{\mathbb{R}^2} \rho_\delta(x-y) \nabla^\perp \phi^n(y) dy = \int_{\mathbb{R}^2} \nabla^\perp \rho_\delta(x-y) \phi^n(y) dy.
\end{equation*}
Since $\nabla^\perp \rho_\delta$ is divergence-free, it follows that so is $W^\delta_n$. 

Using Green's theorem, together with the observation $\nabla_x K(x - y) = - \nabla_y K(x-y)$, we obtain
\begin{eqnarray*}
\int_{|x-y| \geq \epsilon} \nabla_x K(x - y) W_n^\delta(y) \omega^n(y) dy  =  && \int_{|x - y| \geq \epsilon} K(x-y) \diver (\omega^n(y) W_n^\delta(y)) dy - \\ 
& - & \int_{|x-y| = \epsilon} K(x -y) \bigg[ \omega^n(y) W_n^\delta(y) \cdot \frac{x-y}{\epsilon} \bigg] dS(y). 
\end{eqnarray*}
For the first integral, we use the fact that $W_n^\delta$ is divergence-free: 
\begin{equation*}
\int_{|x - y| \geq \epsilon} K(x-y) \diver (\omega^n(y) W_n^\delta(y)) dy = \int_{|x - y| \geq \epsilon} K(x-y) \ W_n^\delta(y) \cdot \nabla \omega^n(y) dy.
\end{equation*}
Then, taking $\delta \rightarrow 0$ on both sides and using the dominated convergence theorem, the equation becomes 
\begin{eqnarray*}
\int_{|x-y| \geq \epsilon} \nabla_x K(x - y) W_n(y) \omega^n(y) dy  =  && \int_{|x - y| \geq \epsilon} K(x-y) W_n(y) \cdot \nabla \omega(y) dy - \\ 
& - & \int_{|x-y| = \epsilon} K(x -y) \bigg[ \omega^n(y) W_n(y) \cdot \frac{x-y}{\epsilon} \bigg] dS(y). 
\end{eqnarray*}
Note that $\nabla \omega^n(y) = H_n'(\phi(y)) \nabla \phi(y)$, so $W_n(y) \cdot \nabla \omega^n(y) = 0$, and the first integral vanishes. For the second, we change coordinates to $z = \frac{x-y}{\epsilon}$ and use the fact that $K$ is homogeneous of degree $-1$: 
\begin{equation*}
 \int_{|x-y| = \epsilon} K(x -y) \bigg[ \omega^n(y) W_n(y) \cdot \frac{x-y}{\epsilon} \bigg] dS(y) = \int_{|z| = 1} K(z) \omega^n(x - \epsilon z) W_n(x - \epsilon z) \cdot z dS(z).
\end{equation*}
Taking now $\epsilon \rightarrow 0$ and denoting the $i$-th component of $W_n$ by $W_n^i$: 
\begin{eqnarray*}
\lim_{\epsilon \rightarrow 0} \int_{|z| = 1} K(z) \omega^n(x - \epsilon z) W_n(x - \epsilon z) \cdot z dS(z)& =& \sum_{i=1}^2 \omega_n(x) W_n^i(x) \int_{|z| = 1} K(z) z_i dS(z) \\ 
& = & \frac{1}{2} \begin{bmatrix}
0 & -1 \\
1 & 0
\end{bmatrix} \omega^n(x) W_n(x). 
\end{eqnarray*}
The conclusion follows. 
\end{proof}

Proposition \ref{prop3.3} and lemma \ref{lemma3.3} give the following corollary. 

\begin{corollary} \label{cor3.2}
Let $(\phi^n, v_n)$ be a solution given by proposition \ref{prop2.1} with $n$ sufficiently large such that proposition \ref{prop3.2} applies; and let $W_n = \nabla^\perp \phi^n$. Then, there exists a constant $C_0$, depending only on $\mu$, such that 
\begin{equation*}
|\nabla v_n W_n|_\mu \leq C_0 |W_n|_\mu (|\nabla v_n|_{L^\infty} + |\omega_0|). 
\end{equation*}
\end{corollary}

We now show that the $\mu$-H\"older semi-norm of $\nabla \phi^n(\cdot, t)$ is controlled by $\int_0^t |\nabla v_n(\cdot,s)|_{L^\infty} ds$. The proof is essentially that presented in \cite{BC}, which we repeat for ease of comparison with the results of section 6.  

\begin{proposition} \label{prop3.4}
Let $(\phi^n, v_n)$ be a solution given by proposition \ref{prop2.1}, with $n$ sufficiently large such that proposition \ref{prop3.2} applies. Then, there exists a constant $C$ depending only on $\mu$ such that 
\begin{equation*}
|\nabla \phi^n(\cdot, t)|_\mu \leq |\nabla \phi_0|_\mu \exp \bigg [ C \bigg( |\omega_0| |t| + \int_0^t |\nabla v_n(\cdot, s)|_{L^\infty} ds  \bigg) \bigg].
\end{equation*}
holds for any $t \in \mathbb{R}$.
\end{proposition}

\begin{proof}
Denote $W_n = \nabla^\perp \phi^n$. Let $Y_n(x, t; \tau)$ denote the backward particle trajectories. They return the position at time $t-\tau$ of a particle which at time $t$ is at position $x$, and are defined by the ODE 
\begin{equation*}
\frac{d}{d \tau} Y_n(x, t; \tau) = - v_n(Y_n(x, t; \tau), t-\tau), \,\,\,\,\,\, Y_n(x,t;0) = x. 
\end{equation*}
In particular, we have $Y_n(x, t; t) = X_n^{-1}(x,t)$. 

Arguing as in the proof of proposition \ref{prop2.3}, we obtain:
\begin{eqnarray*}
\frac{d}{d \tau} W_n(Y_n(x,t; \tau), t - \tau) = - \nabla v_n W_n (Y_n(x, t; \tau), t - \tau). 
\end{eqnarray*}
Therefore, integrating from $0$ to $t$,
\begin{equation*}
W_n(X_n^{-1}(x,t), 0) - W_n(x, t) = \int_0^t \nabla v_n W_n (Y_n(x, t; t - \tau), \tau) d \tau, 
\end{equation*}
which implies 
\begin{eqnarray*}
|W_n(x, t) - W_n(y,t)|& \leq & |W_n^0 (X_n^{-1}(x, t)) - W_n^0 (X_n^{-1}(y, t)) + \\
& + &  \int_0^t \big|\nabla v_n W_n (Y_n(x, t; t - \tau), \tau) - \nabla v_n W_n (Y_n(y, t; t -\tau),\tau)\big| d\tau \\
&\leq& |W_n^0|_{\mu} |\nabla X_n^{-1}(\cdot, t)|_{L^\infty}^\mu |x - y|^{\mu} + \\
& + & \int_0^t |\nabla v_n W_n (\cdot, \tau)|_{\mu}|\nabla  Y_n(\cdot, t; t - \tau)|_{L^\infty}^\mu|x - y|^\mu d\tau \\ 
& \leq & |W_n^0|_{\mu} \exp \bigg [\mu \int_0^{t} |\nabla v_n(\cdot, s)|_{L^\infty} ds \bigg] |x - y|^\mu + \\
& +& \int_0^{t} |\nabla v_n W_n(\cdot,  \tau)|_{\mu} \exp \bigg [ \mu \int_\tau^t |\nabla v_n(\cdot, s)|_{L^\infty} ds \bigg] |x - y|^{\mu} d\tau,
\end{eqnarray*}
where for the last inequality, we used that, by taking $\nabla$ in the ODE defining $Y(x, t; \tau)$ and applying Gr\"onwall's lemma, we obtain  
\begin{equation*}
|\nabla Y_n(\cdot, t; t-\tau)|_{L^\infty} \leq \exp \bigg[\int_\tau^t |\nabla v_n(\cdot, s)|_{L^\infty} ds \bigg].
\end{equation*}
Using corollary \ref{cor3.2}, 
\begin{eqnarray*}
|W_n(\cdot, t)|_\mu & \leq & |W_n^0|_\mu \exp \bigg [\mu \int_0^{t} |\nabla v_n(\cdot, s)|_{L^\infty} ds \bigg] + \\
& + & C_0 \int_0^t (|\nabla v_n(\cdot, \tau)|_{L^\infty} + |\omega_0|) |W_n(\cdot, \tau)|_\mu \exp \bigg [ \mu \int_\tau^t |\nabla v_n(\cdot, s)|_{L^\infty} ds  \bigg].
\end{eqnarray*}
We denote $G(t) = |W_n(\cdot, t)|_\mu \exp \big[ - \mu \int_0^t |\nabla v_n(\cdot, s)|_{L^\infty} ds \big]$ and multiply the last inequality by $\exp \big[ - \mu \int_0^t |\nabla v_n(\cdot, s)|_{L^\infty} ds \big]$ to obtain
\begin{equation*}
G(t) \leq G(0) + C_0 \int_0^t (|\nabla v_n(\cdot, \tau)|_{L^\infty} + |\omega_0|) G(\tau) d\tau.
\end{equation*}
Together with Gr\"onwall's lemma, this implies 
\begin{equation*}
|W(\cdot, t)|_\mu \leq |W_0^n|_\mu \exp \bigg [ C_0 |\omega_0|t + (C_0 + \mu) \int_0^t |\nabla v_n(\cdot, s)|_{L^\infty} ds  \bigg],
\end{equation*}
which concludes the proof.
\end{proof}

Propositions \ref{prop2.3}, \ref{prop3.2}, and \ref{prop3.4} give the desired uniform bounds, which we express in the following corollary. 

\begin{corollary} \label{cor3.3}
Let $(\phi^n, v_n)$ be a solution given by proposition \ref{prop2.1}, and let $n \geq M$ with $M$ sufficiently large such that proposition \ref{prop3.2} applies. Then, there exists a constant $C$ depending only on $\phi_0$ and $\mu$ (in particular, independent of time or $n$) such that for any time $t \in \mathbb{R}$ we have
\begin{equation} \label{eq3.11}
|\nabla v_n|_{L^\infty} \leq C |\omega_0| \exp \bigg[C |\omega_0| |t|    \bigg]
\end{equation}
\begin{equation*}
|\nabla \phi^n(\cdot, t)|_{M-\inf} \geq |\nabla \phi_0|_{M-\inf} \exp \bigg[ \exp \big[ - C |\omega_0| |t| \big] \bigg]
\end{equation*}
\begin{equation*}
|\nabla \phi^n(\cdot, t)|_{L^\infty} \leq |\nabla \phi_0|_{L^\infty} \exp \bigg[ \exp \big[ C |\omega_0| |t| \big] \bigg]
\end{equation*}
\begin{equation*}
|\nabla \phi^n(\cdot, t)|_\mu \leq |\nabla \phi_0|_{\mu} \exp \bigg[ \exp \big[C |\omega_0| |t| \big] \bigg]
\end{equation*}
\end{corollary}

\begin{proof}
Inequality \eqref{eq3.11} implies the rest by direct calculation. For \eqref{eq3.11}, we have from proposition \ref{prop3.2}, that, if $L/\Delta_n < 1$ then 
\begin{equation*}
|\nabla v_n|_{L^\infty} \leq C |\omega_0|,
\end{equation*}
which implies that \eqref{eq3.11} holds in this case. If $L/ \Delta_n \geq 1$, note that $\log(1 + x) \leq 1 + \log x$ for $x \geq 1$, so 
\begin{equation*}
|\nabla v_n|_{L^\infty} \leq C |\omega_0| \bigg(1 + \log \bigg(\frac{L}{\Delta_n} \bigg) \bigg). 
\end{equation*}
Proposition \ref{prop3.4} and \ref{prop2.3} imply that 
\begin{equation*}
\frac{L}{\Delta_n(t)} \leq \frac{L}{\Delta_n(0)} \exp \bigg[ C|\omega_0| |t| + C \int_0^t |\nabla v(\cdot, s)|_{L^\infty} ds \bigg].
\end{equation*}
Therefore, 
\begin{equation*}
|\nabla v_n|_{L^\infty} \leq C |\omega_0| \bigg(1 + \log(L/ \Delta_n(0))  + |\omega_0|| t| + \int_0^t |\nabla v_n(\cdot, s)|_{L^\infty} ds \bigg).
\end{equation*}
But, by definition
\begin{equation*}
\Delta_n(0) = \bigg(\frac{|\nabla \phi_0|_{M-\inf}}{|\nabla \phi_0|_{\mu}} \bigg)^{\frac{1}{\mu}},
\end{equation*}
so $\Delta_n(0)$ depends only on $\phi_0$. Therefore, we can absorb $\log(L / \Delta_n(0))$ in the constant
\begin{equation*}
|\nabla v_n|_{L^\infty} \leq C |\omega_0| \bigg(1  + |\omega_0| |t| + \int_0^t |\nabla v_n(\cdot, s)|_{L^\infty} ds \bigg).
\end{equation*}
Gr\"onwall's lemma concludes the proof.
\end{proof}

\section{Existence of \texorpdfstring{$C^{1,\, \mu}$}{C(1, mu)}  Vortex Patches}

In this section, we use the uniform bounds to find a subsequence converging to a function $\phi$. We then show that $\phi$ solves equations \eqref{eq1.1}-\eqref{eq1.3}. The idea of the proof is similar to the one presented in \cite{MB} for the existence of Yudovich weak solutions, and some of the necessary inequalities are proved here in the same way.

\begin{proposition} \label{prop4.1}
Let $(\phi^n, v_n)$ be solutions given by proposition \ref{prop2.1} with $n \in \mathbb{Z}_+$. Then, there exists a subsequence, which we still denote $(\phi^n, v_n)$, and a function $\phi:\mathbb{R}^2 \times \mathbb{R} \rightarrow \mathbb{R}$, $\phi (\cdot, t) \in C^{1, \, \mu}(\mathbb{R}^2)$, $|\nabla \phi(\cdot, t)|_{\inf} > 0$, $\phi(x,0) = \phi_0(x)$ such that \\
(i) $\phi^n$ and $\nabla \phi^n$ converge uniformly on compact sets of $\mathbb{R}^2 \times \mathbb{R}$ to $\phi$ and $\nabla \phi$, respectively; \\
(ii) If we define $\omega(x,t) = \omega_0 H(\phi(x,t))$, then for each time $t \in \mathbb{R}$, $\omega^n(\cdot, t)$ converges to $\omega (\cdot, t)$ in $L^1$. \\
(iii) If we define $v$ as in \eqref{eq1.3}, then $v$ is continuous in both space and time and, for each time $t \in \mathbb{R}$, $v_n(\cdot, t)$ converge uniformly to $v(\cdot, t)$. More precisely, there exists a constant $C > 0$ such that for any $\epsilon > 0$, 
\begin{equation} \label{eq4.1}
|v(\cdot, t) - v_n(\cdot, t)|_{L^\infty} \leq C\bigg[ \epsilon + \frac{1}{\epsilon}|\omega^n(\cdot, t) - \omega(\cdot, t)|_{L^1}\bigg].
\end{equation}
\end{proposition} 

\begin{proof}
We have seen in the proof of proposition \ref{prop3.4} that 
\begin{equation*}
|\nabla X_n^{-1}(\cdot, t)|_{L^\infty} \leq \exp \bigg[ \int_0^t |\nabla v_n(\cdot, s)|_{L^\infty} ds \bigg].
\end{equation*}
Therefore, by corollary \ref{cor3.3}, we have that on any time interval $[-T, T]$, $|\nabla  X_n^{-1}(\cdot, t)|_{L^\infty}$ are uniformly bounded for sufficiently large $n \in \mathbb{Z}_+$. The mean value theorem implies that $X_n^{-1}(x, t)$ are uniformly Lipschitz in the space variables. We show also that $X_n^{-1}$ are Lipschitz in time. For $-T \leq t_1 \leq t_2 \leq T$, let $\alpha = Y_n(x, t_2; t_2 - t_1)$, where $Y_n$ is defined as in the proof of proposition \ref{prop3.4}. Therefore, $X_n^{-1}(x, t_2) = X_n^{-1}(\alpha, t_1)$. Note that by proposition \ref{prop2.2}, there exists a constant $C$ independent of time or $n$, such that
\begin{equation*}
|x-\alpha| = \bigg | \int_{t_1}^{t_2} v_n (X_n(\alpha, \tau - t_1), \tau) d \tau \bigg| \leq C |\omega_0| |t_2 - t_1|. 
\end{equation*}  
Therefore, 
\begin{eqnarray*}
|X_n^{-1}(x, t_2) - X_n^{-1}(x, t_1)| &\leq & |X_n^{-1}(\alpha, t_1) - X_n^{-1}(x, t_1)| \\
& \leq & |\nabla X_n^{-1}|_{L^\infty(\mathbb{R}^2 \times [-T, T])}|\alpha - x| \\
& \leq & C |\omega_0| |\nabla X_n^{-1}|_{L^\infty(\mathbb{R}^2 \times [-T, T])}|t_1 - t_2|.
\end{eqnarray*}
It follows that $X_n^{-1}$ are equicontinuous and bounded on $B_R\times [-T, T]$ for any $R \in \mathbb{Z}_+$ and $T\in \mathbb{Z}_+$. So, by the Arzel\`a-Ascoli Theorem, there exists a subsequence which converges uniformly on each such $B_R\times [-T, T]$. By a standard diagonalization argument, we obtain that there exists a subsequence which converges uniformly on compact sets to a continuous function which we denote by $X^{-1}(x,t)$.  We show now that $\nabla \phi^n$ are also equicontinuous and uniformly bounded. The fact that they are equicontinuous in the space variables and bounded on any time interval $[-T, T]$ follows from corollary \ref{cor3.3}. It remains to check that they are equicontinuous in the time variables. Let $-T \leq t_1 \leq t_2 \leq T$, and denote, as before, $W_n = \nabla^\perp \phi^n$. Fix $\alpha \in \mathbb{R}^2$ such that $x = X_n(\alpha,t_1)$. Then,
\begin{eqnarray*}
|W_n(x, t_1) - W_n(x,t_2)| &=& |W_n(X_n(\alpha, t_1), t_1) - W_n(X_n(\alpha, t_1), t_2)| \\
& \leq & |W_n(X_n(\alpha, t_1), t_1) - W_n(X_n(\alpha, t_2), t_2)| + \\
& & + |W_n(X_n(\alpha, t_1), t_2) - W_n(X_n(\alpha, t_2), t_2)| \\ 
&=& T_1 + T_2.
\end{eqnarray*}
For $T_1$, by the mean value theorem,
\begin{eqnarray*}
T_1 &\leq & \sup_{t \in [t_1, t_2]} \big|\frac{d}{dt} W_n(X_n(\alpha, t), t)\big| |t_1 - t_2| \\
& \leq & |\nabla_\alpha X_n|_{L^\infty(\mathbb{R}^2 \times [-T, T])} |\nabla \phi^n|_{L^\infty(\mathbb{R}^2 \times [-T, T])} |t_1 - t_2|, 
\end{eqnarray*}
where we used proposition \ref{prop2.3}. Also, by applying $\nabla_\alpha$ to the ODE defining $X_n$, \eqref{eq2.8}, we obtain by Gr\"onwall's lemma that $\nabla_\alpha X_n$ are indeed uniformly bounded on $\mathbb{R}^2 \times [-T, T]$. For $T_2$, we have 
\begin{eqnarray*}
T_2 &\leq & |W_n(\cdot, t_2)|_\mu |X_n(\alpha, t_1) - X_n(\alpha, t_2)|^\mu \\ 
& \leq & |W_n(\cdot, t_2)|_\mu |v_n|^\mu_{L^\infty(\mathbb{R}^2 \times [-T, T])} |t_1 - t_2|^\mu.
\end{eqnarray*}
In view of corollary \ref{cor3.3} and proposition \ref{prop2.2}, we arrive at the conclusion that $\nabla \phi^n$ are uniformly bounded and equicontinuous on every set $B_R \times [-T,T]$, so, arguing as before, we find a further subsequence that converges uniformly on compact sets to a continuous mapping $\psi:\mathbb{R}^2 \times \mathbb{R} \rightarrow \mathbb{R}^2$. Note that by construction, $\psi(\cdot, t) \in  C^{0, \, \mu}(\mathbb{R}^2)$. 

Let $\phi(x,t) = \phi_0(X^{-1}(x,t))$. We contend that this is our wanted function. Note that since $X_n^{-1}$ converge uniformly to $X^{-1}$ on compact sets and $\phi_0$ is, by the mean value theorem, Lipschitz (so in particular uniformly continuous), it follows that $\phi^n(x,t) = \phi_0(X_n^{-1}(x,t))$ converge uniformly on compact sets to $\phi(x,t)$. Clearly, then, $\phi(x,0) = \phi_0(x)$. We claim that $\psi = \nabla \phi$. To prove this, since $\psi$ is continuous, we only need to show that 
\begin{equation*}
\phi(x + he_i, t) - \phi(x,t) = \int_0^h \psi (x + se_i, t) \cdot e_i ds,
\end{equation*}
for $i = 1,2$, where $\{e_1, e_2\}$ is the standard basis of $\mathbb{R}^2$. Note that we already have 
\begin{equation*}
\phi^n(x + he_i, t) - \phi^n(x,t) = \int_0^h \nabla \phi^n (x + se_i, t) \cdot e_i ds,
\end{equation*}
and the wanted result follows by observing that $\phi^n(x,t) \rightarrow \phi(x,t)$ as $n \rightarrow \infty$ and 
\begin{equation*}
\int_0^h \nabla \phi^n (x + se_i, t) \cdot e_i ds \rightarrow \int_0^h \psi (x + se_i, t) \cdot e_i ds, 
\end{equation*}
since $\nabla \phi^n$ converge uniformly on compacts. Therefore, (i) is proved and, moreover, we have by construction that $\phi (\cdot, t) \in C^{1, \, \mu}(\mathbb{R}^2)$. 

We now show that for each $t \in \mathbb{R}$, $|\nabla \phi(\cdot, t)|_{\inf} > 0$. Let $x \in \mathbb{R}^2$ such that $\phi(x,t) = 0$. Then, if we denote $X^{-1}(x,t) = \alpha$, we have that $\alpha \in \phi_0^{-1}(\{0\})$. Since $|\nabla \phi_0 (\alpha)| \geq m > 0$ by assumption, there exists $\delta > 0$ such that $|\nabla \phi_0(\tilde{\alpha})| \geq m/2$ for all $|\tilde{\alpha} - \alpha| \leq \delta$.  Since $\alpha_n := X_n^{-1}(x,t) \rightarrow X^{-1}(x,t) = \alpha$ as $n \rightarrow \infty$, we have that for all $n$ sufficiently large $|\nabla \phi_0(\alpha_n)| \geq m/2$. On the other hand, we have seen in the proof of proposition \ref{prop2.3} that 
\begin{equation*}
|\nabla \phi^n (X_n(\alpha, t), t)| \geq |\nabla \phi_0(\alpha)| \exp \bigg[ - \int_0^t |\nabla v_n(\cdot, s)|_{L^\infty}ds \bigg].
\end{equation*}
Plugging $\alpha_n$ into the inequality above and noting that the $|\nabla v_n(\cdot, s)|_{L^\infty}$ are uniformly bounded by corollary \ref{cor3.3}, we obtain that there exists a constant depending on time, but not on $n$, $C = C(t) > 0$ such that 
\begin{equation*}
|\nabla \phi^n(x, t)| \geq C(t) m > 0.
\end{equation*}
Since $\nabla \phi^n(x,t) \rightarrow \nabla \phi(x,t)$, it follows that, by taking the limit $n \rightarrow \infty$,
\begin{equation*}
|\nabla \phi(x,t)| \geq C(t) m > 0.
\end{equation*}
Since this holds for all $x$ such that $\phi(x,t) = 0$, we conclude that $|\nabla \phi(\cdot, t)|_{\inf} > 0$ for all times $t \in \mathbb{R}$. 

To prove (ii), we claim that the following holds: for any $t \in \mathbb{R}$, and all $f \in L^1(\mathbb{R}^2)$, 
\begin{equation*}
\int_{\mathbb{R}^2} f(X^{-1}(x,t)) dx = \int_{\mathbb{R}^2} f(x) dx. 
\end{equation*}
Indeed, if $f \in C_0^\infty(\mathbb{R}^2)$, then $f(X_n^{-1}(x,t)) \rightarrow f(X^{-1}(x,t))$ pointwise. Moreover $|f(X_n^{-1}(x,t))| \leq |f|_{L^\infty}$, which is integrable since $f$ has compact support. Then, by the dominated convergence theorem, $|f(X_n^{-1}(x,t)) - f(X^{-1}(x,t))|_{L^1(\mathbb{R}^2)} \rightarrow 0$. Since $\int f(X_n^{-1}(x,t))dx = \int f(x) dx$, the assertion follows for $f \in C_0^\infty$, and thus, also in $L^1$ by density. Note that if we take $f$ to be a characteristic function, we obtain that $X(\cdot, t)$ is measure-preserving. We have 
\begin{equation*}
|\omega(\cdot, t) - \omega^n(\cdot, t)|_{L^1} \leq |\omega_0||H(\phi^n(\cdot, t)) - H_n(\phi^n(\cdot, t))|_{L^1} + |\omega_0||H(\phi^n(\cdot, t)) - H(\phi(\cdot, t))|_{L^1}.  
\end{equation*}
For the first term,
\begin{eqnarray*}
|H \circ \phi_0(X_n^{-1}(\cdot, t)) - H_n \circ \phi_0 (X_n^{-1}(\cdot, t))|_{L^1} = |H \circ \phi_0 - H_n \circ \phi_0|_{L^1}.
\end{eqnarray*}
It is clear by the definition of $H_n$ that it converges to $H$ everywhere except at $0$. Since $\partial D_0 = \phi_0^{-1}(\{0\})$ is a compact $C^{1}$ submanifold of the plane, it follows by Sard's theorem that $m(\partial D_0) = 0$, and so $H_n\circ \phi_0$ converges almost everywhere to $H \circ \phi_0$. By the dominated convergence theorem, we obtain that the first term goes to zero as $n \rightarrow \infty$. For the second term, let $\delta > 0$ and $H_\delta \in C_0^\infty(\mathbb{R}^2)$ such that $|H_\delta - H\circ \phi_0|_{L^1} \leq \delta$ (which exists since clearly $H\circ \phi_0$ is in $L^1$). Then, 
\begin{eqnarray*}
|H\circ \phi_0 (X_n^{-1}(\cdot, t)) - H \circ \phi_0 (X^{-1}(\cdot, t))|_{L^1} & \leq & |H \circ \phi_0(X_n^{-1}(\cdot, t)) - H_\delta(X_n^{-1}(\cdot, t))|_{L^1} + \\
&+&  |H_\delta (X_n^{-1}(\cdot, t)) - H_\delta(X^{-1}(\cdot, t))|_{L^1} + \\ 
&+&  |H \circ \phi_0(X^{-1}(\cdot, t)) - H_\delta(X^{-1}(\cdot, t))|_{L^1} \\ 
& \leq & 2 \delta + |H_\delta (X_n^{-1}(\cdot, t)) - H_\delta(X^{-1}(\cdot, t))|_{L^1},
\end{eqnarray*}
where for the last inequality we used the claim. The remaining term goes to zero as $n \rightarrow \infty$ by the uniform convergence of $X_n^{-1}$ to $X^{-1}$ on compact sets and the fact that $H_\delta$ is compactly supported and smooth. Since $\delta$ was chosen arbitrarily, we conclude that (ii) holds. 

It remains to prove (iii). Let $\rho \in C_0^\infty(\mathbb{R}^2)$ be a cut-off function $\rho(x) = 1$ when $|x| \leq 1$, and $\rho(x) = 0$ when $|x| \geq 2$, $0 \leq \rho(x) \leq 1$. For $\epsilon > 0$, let $\rho_\epsilon(x) = \rho(\frac{x}{\epsilon})$. Then, 
\begin{eqnarray*}
|v_n(x,t) - v(x,t)| \leq |K\rho_{\epsilon} * [\omega - \omega^n]| + |K(1 - \rho_\epsilon)*[\omega - \omega^n]| = T_1 + T_2. 
\end{eqnarray*}
Using (the trivial case of) Young's inequality, we have 
\begin{equation*}
T_1 \leq |K\rho_\epsilon|_{L^1} |\omega - \omega^n|_{L^\infty} \leq 2 |\omega_0| |K\rho_\epsilon|_{L^1}.
\end{equation*}
But 
\begin{equation*}
|K\rho_\epsilon|_{L^1} \leq \frac{1}{2\pi} \int_{|x| \leq 2\epsilon} \frac{1}{|x|} dx = 2 \epsilon.
\end{equation*}
Young's inequality applied to $T_2$ gives 
\begin{equation*}
T_2 \leq |K (1 - \rho_\epsilon)|_{L^\infty} |\omega - \omega^n|_{L^1} \leq \frac{1}{2\pi} \frac{1}{\epsilon} |\omega - \omega_n|_{L^1},
\end{equation*}
and \eqref{eq4.1} follows. Therefore, $v_n(\cdot, t)$ converge uniformly to $v(\cdot, t)$ and by the uniform convergence theorem, we have that $v$ is continuous in the space variables. To show that $v$ is continuous also in the time variables, we show that $v(x, t') \rightarrow v(x,t)$ when $t' \rightarrow t$. Fix $x \in \mathbb{R}^2$. Then $v(x,t) = \int_{\mathbb{R}^2} u_t(y) dy$, where 
\begin{equation*}
u_t(y) = K(x-y) H(\phi(y,t)). 
\end{equation*}
It suffices to show that $u_{t'} \rightarrow u_{t}$ in $L^1$ as $t' \rightarrow t$. We claim that the convergence holds almost everywhere. Indeed, if $\phi(y, t) \neq 0$, since $\phi$ is continuous in time, we have that for all $t'$ close to $t$, $\phi(y, t') \neq 0$ and has the same sign as $\phi(y,t)$. Therefore $H(\phi(y,t)) = H(\phi(y, t'))$ which implies $u_t(y) = u_{t'}(y)$ for all $t'$ close to $t$. On the other hand, since $|\nabla \phi(\cdot, t)|_{\inf} > 0$ and $\phi(\cdot, t) \in C^1(\mathbb{R}^2)$, it follows by the inverse mapping theorem that $\phi^{-1}(\{0\})$ is a $C^1$ 1-dimensional submanifold of $\mathbb{R}^2$, and thus, by Sard's theorem, it has measure zero. Therefore, $u_{t'} \rightarrow u_t$ almost everywhere as $t' \rightarrow t$. Also, we have already seen that $X(\cdot, t)$ is measure preserving, so 
\begin{equation*}
m(\{\phi(x,t) \geq 0\}) = m(\{\phi_0(x) \geq 0\}) = m(D_0), 
\end{equation*}
and, therefore, arguing as in proposition \ref{prop2.2}, we see that $u_t$ is dominated by an integrable function. Thus, by the dominated convergence theorem, we conclude that $v(x, t') \rightarrow v(x, t)$ which implies continuity in time, and concludes the proof of the third assertion. 
\end{proof}

\begin{remark}
Arguing as we did for $X_n^{-1}(x,t)$, we can also assume, by passing to yet another subsequence if necessary, that $X_n(\alpha, t)$ also converge uniformly on compact sets to a mapping $X(\alpha, t)$. It is not difficult to check that $X(\alpha, t)$ and $X^{-1}(x,t)$ are inverse homeomorphisms. 
\end{remark}

We are now ready to prove the main result of this section, theorem \ref{thm1}. 

\begin{proof}[Proof of theorem \ref{thm1}]
Let $\phi$, $\omega$, and $v$ be given by proposition \ref{prop4.1}. By definition, $v$ satisfies \eqref{eq1.3} and it was already noted that \eqref{eq1.2} is also satisfied. So, it remains to show that $\partial_t \phi$ exists and it satisfies \eqref{eq1.1}. By proposition \ref{prop4.1}, this will also imply its continuity since both $v$ and $\nabla \phi$ are continuous in time and space. Let $-T \leq t_1 \leq t_2 \leq T$ and note that \eqref{eq2.1} implies 
\begin{equation*}
\phi^n(x, t_2) - \phi^n(x,t_1) = - \int_{t_1}^{t_2} v_n (x,t) \cdot \nabla \phi^n (x,t) dt. 
\end{equation*}
By proposition \ref{prop4.1}, $\phi^n(x,t) \rightarrow \phi(x,t)$ as $n \rightarrow \infty$. We claim that 
\begin{equation*}
\int_{t_1}^{t_2} v_n (x,t) \cdot \nabla \phi^n (x,t)dt \rightarrow \int_{t_1}^{t_2} v(x,t) \cdot \nabla \phi (x,t) dt. 
\end{equation*}
Indeed, we write 
\begin{eqnarray*}
\int_{t_1}^{t_2} |v_n (x,t) \cdot \nabla \phi^n (x,t) - v(x,t) \cdot \nabla \phi (x,t)| dt &\leq& \int_{t_1}^{t_2} |v_n(\cdot, t)|_{L^\infty} |\nabla \phi_n(x,t) - \nabla \phi(x,t)| dt + \\ 
& + & \int_{t_1}^{t_2} |\nabla \phi(\cdot, t)|_{L^\infty} |v_n(\cdot, t) - v(\cdot, t)|_{L^\infty} dt \\
&=& T_1 + T_2. 
\end{eqnarray*}
The first term goes to zero as $n \rightarrow \infty$ since $|v_n(\cdot, t)|_{L^\infty}$ are bounded independently of $n$ by proposition \ref{prop2.2}, and $|\nabla \phi^n - \nabla \phi|$ converges uniformly to zero on compact sets. For the second term, by corollary \ref{cor3.3} and the convergence of $\nabla \phi^n$ to $\nabla \phi$, we have that there exists a constant $C(T)$ depending on $T$ and $\phi_0$ such that $|\nabla \phi (\cdot, t)|_{L^\infty} \leq C(T)$; and by proposition \ref{prop4.1} (iii), we have that for any $\epsilon > 0$, $|v_n(\cdot, t) - v(\cdot, t)|_{L^\infty} \leq C\epsilon + (C/\epsilon) |\omega(\cdot, t) - \omega_n(\cdot,t)|_{L^1}$. Therefore, 
\begin{equation*}
T_2 \leq C\epsilon (t_2 - t_1) + \frac{C}{\epsilon} \int_{t_1}^{t_2} |\omega(\cdot, t) - \omega^n(\cdot, t)|_{L^1} dt.
\end{equation*}
We know that the measurable functions $|\omega(\cdot, t) - \omega^n(\cdot, t)|_{L^1}$ converge to zero for all $t \in [t_1, t_2]$ and 
\begin{equation*}
|\omega(\cdot, t) - \omega_n(\cdot, t)|_{L^1} \leq |\omega(\cdot, t)|_{L^1} + |\omega^n(\cdot, t)|_{L^1}  \leq |\omega_0|\big[|H \circ \phi_0|_{L^1} + |H_n \circ \phi_0|_{L^1} \big] \leq 2 |\omega_0| m(D_0),
\end{equation*}
which is integrable on $[t_1, t_2]$. Therefore, the dominated convergence theorem assures that the integral goes to zero as $n \rightarrow \infty$. We have obtained that 
\begin{equation*}
\limsup_{n \rightarrow \infty} \int_{t_1}^{t_2} |v_n (x,t) \cdot \nabla \phi^n (x,t) - v(x,t) \cdot \nabla \phi (x,t)| dt \leq C \epsilon (t_2 - t_1).
\end{equation*}
Since $\epsilon>0$ was chosen arbitrarily, the claim follows, and thus 
\begin{equation*}
\phi(x,t_2) - \phi(x,t_1) = - \int_{t_1}^{t_2} v(x,t) \cdot \nabla \phi(x,t) dt.  
\end{equation*}
Since the integrand is continuous in time, we have that $\partial_t \phi$ exists and 
\begin{equation*}
\partial_t \phi = - v \cdot \nabla \phi,
\end{equation*}
concluding the proof of the theorem. 
\end{proof}

\begin{remark}
The pair $(\omega, v)$ given by proposition \ref{prop4.1} is the unique Yudovich weak solution with initial vorticity $\omega(x,0) = H(\phi_0(x))$, since its construction is essentially that in the proof of the existence part of Yudovich's theorem (see, for example, pages 312-313 in \cite{MB}).
\end{remark}

\section{Expressions for the Higher Derivatives of the Tangent Vectors}

We now begin the proof of the higher regularity of the boundary $\partial D$ in the case when $\phi_0$ is a $C^{k,\, \mu}$ initial scalar. In order to obtain inductive expressions for the higher derivatives that are of the same flavor as \eqref{eq2.10} and \eqref{eq3} and for which we can perform a similar analysis, we switch from the global formulation given by equations \eqref{eq1.1}-\eqref{eq1.3} to one centered on the boundary. In this section, we show that the H\"older regularity of the boundary is controlled by the $C^{0, \, \mu}$ norms of terms of the form $(\nabla^\perp \phi \cdot \nabla)^j \nabla^\perp \phi $, $j < k$. We will continue using the notation $W_0 = \nabla^\perp \phi_0$ and $W_n = \nabla^\perp \phi^n$.

If $\phi_0 \in C^{k,\,\mu}$, $k \in \mathbb{Z}_+$, $D_0$ is bounded and simply connected, and $|\nabla \phi_0|_{\inf} \geq m > 0$, then $\partial D_0$ is a simple closed curve in the plane which we can parametrize to obtain $z_0:\mathbb{S}^1 \rightarrow \mathbb{R}^2$, with  tangent vector 
\begin{equation}\label{eq5.1}
\frac{dz_0}{d\alpha}(\alpha) = f_0(\alpha) W_0(z_0(\alpha)),
\end{equation} 
where $f_0$ is a positive, real-valued function in $C^{k-1, \, \mu}(\mathbb{S}^1)$ and $W_0 = \nabla^\perp \phi_0$ as before. The boundary $\partial D_n$ is transported by the particle trajectories such that at time $t$, it is given by 
\begin{equation} \label{eq5.2}
z_n(\alpha, t) = X_n(z_0(\alpha), t).
\end{equation}
Then, 
\begin{equation*}
\frac{d}{dt}z_n(\alpha, t) = v_n(z_n(\alpha, t), t). 
\end{equation*}
Similarly, if $X(\alpha, t)$ is the particle-trajectory mapping of the vortex patch given by theorem \ref{thm1} (see remark after proposition \ref{prop4.1}) then a parametrization of the boundary of the patch is
\begin{equation} \label{eq5}
z(\alpha, t) = X(z_0(\alpha), t).
\end{equation}

In order to state the next proposition, let us introduce some notation. If $u, v: \mathbb{R}^2 \rightarrow \mathbb{R}^2$, then for $m \in \mathbb{N}$, we write $(u\cdot \nabla)^m v$ to mean the $m$ times application of the operator $u \cdot \nabla$ on $v$: 
\begin{equation*}
(u \cdot \nabla)^0 v  = v 
\end{equation*}
\begin{equation*}
(u \cdot \nabla)^m v  =  u\cdot \nabla [ (u \cdot \nabla)^{m-1} v]
\end{equation*}

\begin{proposition} \label{prop5.1}
Assume that $\phi_0$ is a $C^{k, \, \mu}$ initial scalar. Let $\phi^n$ be a solution given by proposition \ref{prop2.1}, and let $z_n$ and $z_0$ be defined as above. Then for $1 \leq j \leq k$, the $j^{\text{th}}$ derivative of $z_n$, $\frac{d^j}{d\alpha^j} z_n(\alpha, t)$, is a linear combination of terms of the form 
\begin{equation} \label{eq5.3}
f_0^{(j_1)}(\alpha)...f_0^{(j_l)}(\alpha) [(W_n \cdot \nabla)^{j - 1 - \sum_i j_i }W_n] (z_n(\alpha, t), t),
\end{equation}
where $j_i \in \mathbb{N}$, $l \in \mathbb{Z_+}$, $\sum_{i=1}^l j_i \leq j - 1$. 
\end{proposition}

\begin{proof}
For $j = 1$ we differentiate \eqref{eq5.2} to obtain 
\begin{equation*}
\frac{d}{d\alpha}z_n(\alpha, t) = f_0(\alpha) W_0(z_0(\alpha)) \cdot \nabla X_n(z_0(\alpha), t) = f_0(\alpha) \nabla X_n(z_0(\alpha), t) W_0(z_0(\alpha)). 
\end{equation*}
But we have already noted in the proof of proposition \ref{prop2.3} that
\begin{equation*}
W_n(X_n, t) = \nabla X_n W_0,
\end{equation*}
so
\begin{equation*}
\frac{d}{d\alpha}z_n(\alpha, t) = f_0(\alpha) W_n(z_n(\alpha, t), t),
\end{equation*}
which clearly has the wanted form. 

We now argue inductively. Suppose the claim holds for some $1 \leq j \leq k-1$. We differentiate \eqref{eq5.3} and note that if the derivative falls on one of the derivatives of $f_0$, it is easily seen that the term retains its form for $j + 1$; if it falls on the last term, we have
\begin{eqnarray*}
\frac{d}{d\alpha} [(W_n \cdot \nabla)^{j - 1 -\sum_i j_i}W_n](z_n(\alpha, t), t) & = &  \big( \frac{d}{d\alpha}z_n \big) \cdot \nabla [(W_n \cdot \nabla)^{j - 1 -\sum_i j_i}W_n](z_n(\alpha, t), t) \\ 
& = & f_0(\alpha) [(W_n \cdot \nabla)^{j + 1 - 1 - \sum_i j_i}W_n](z_n(\alpha, t), t),
\end{eqnarray*}
where we use the result for $j = 1$ and \eqref{eq5.1}. The conclusion follows.
\end{proof}

Proposition \ref{prop5.1} has the following immediate corollary, which shows that we only need to treat the regularity of $(W_n \cdot \nabla)^j W_n$ in order to obtain uniform bounds for $|z_n|_{C^{k, \, \mu}(\mathbb{S}^1)}$. 

\begin{corollary} \label{cor5.1}
Assume that $\phi_0$ is a $C^{k, \, \mu}$ initial scalar. Let $\phi^n$ be a solution given by proposition \ref{prop2.1}, and let $z_n$ be defined as above. Then, there exists a constant $C$ depending only on $f_0$ (in particular, independent of $n$ or $t$) such that for each $t \in \mathbb{R}$,
\begin{equation*}
|z_n(\cdot, t)|_{C^{k, \, \mu}(\mathbb{S}^1)} \leq C \sum_{j = 0}^{k-1} |(W_n \cdot \nabla)^j W_n (z_n(\cdot, t), t)|_{C^{0, \, \mu}(\mathbb{S}^1)}. 
\end{equation*}
\end{corollary}

We end this section with an extension of \eqref{eq2.10}. 

\begin{proposition} \label{prop5.2}
Assume $\phi_0$ is a $C^{k, \, \mu}$ initial scalar and let $(\phi^n, v_n)$ be a solutions given by proposition \ref{prop2.1}. Then for each $0 \leq j \leq k-1$,
\begin{equation} \label{eq5.4}
\frac{d}{dt} (W_n \cdot \nabla)^j W_n (X_n(\alpha, t), t) = (W_n \cdot \nabla)^{j+1} v_n (X_n(\alpha, t), t). 
\end{equation}
\end{proposition}

\begin{proof}
We argue inductively. The case $j = 0$ was proved in proposition \ref{prop2.3}, where we also noted that 
\begin{equation*}
W_n(X_n(\alpha, t), t) = \nabla_\alpha X_n(\alpha, t) W_0(\alpha). 
\end{equation*}
Suppose the result holds for $j<k-1$. We have
\begin{equation*}
\nabla_\alpha[(W_n \cdot \nabla)^j W_n (X_n(\alpha, t),t)] = \nabla[(W_n\cdot \nabla)^j W_n](X_n(\alpha, t),t) \nabla_\alpha X_n(\alpha, t)
\end{equation*}
Applying the matrix above to the vector $W_0(\alpha)$, we obtain
\begin{eqnarray*}
\nabla_\alpha[(W_n \cdot \nabla)^j W_n (X_n(\alpha, t),t)] W_0(\alpha)& =& \nabla[(W_n\cdot \nabla)^j W_n] W_n (X_n(\alpha, t), t) \\
& = & (W_n \cdot \nabla)^{j+1} W_n (X_n(\alpha, t), t). 
\end{eqnarray*}
Now, we differentiate with respect to time, commute $\nabla_ \alpha$ with $\frac{d}{dt}$, and use the inductive hypothesis: 
\begin{eqnarray*}
\frac{d}{dt} (W_n \cdot \nabla)^{j+1} W_n(X_n(\alpha, t), t) &=& \nabla_\alpha [(W_n \cdot \nabla)^{j+1} v_n(X_n(\alpha, t),t)] W_0(\alpha) \\ 
& = & \nabla[(W_n \cdot \nabla)^{j+1} v_n](X_n(\alpha, t),t) \nabla_\alpha X_n(\alpha, t) W_0(\alpha) \\ 
& = & (W_n \cdot \nabla)^{j+2} v_n (X_n(\alpha, t), t). 
\end{eqnarray*}
This concludes the proof. 
\end{proof}

\begin{corollary} \label{cor5.2}
Assume $\phi_0$ is a $C^{k, \, \mu}$ initial scalar. Let $(\phi^n, v_n)$ be a solution given by proposition \ref{prop2.1}, and let $z_n$ be defined as above. Then, there exists a constant $C$ depending only on $f_0$ (in particular, independent of $n$ or $t$) such that for each $t \in \mathbb{R}$,
\begin{equation*}
\bigg| \frac{d}{dt} \frac{d^k}{d \alpha^k}z_n(\cdot, t) \bigg|_{L^\infty(\mathbb{S}^1)} \leq C \sum_{j=1}^{k}|(W_n \cdot \nabla)^j v_n(\cdot, t)|_{L^\infty}
\end{equation*}
\end{corollary}

\begin{proof}
Differentiating each term given by proposition \ref{prop5.1} with respect to time and using proposition \ref{prop5.2}, we obtain that $\frac{d}{dt} \frac{d^k}{d \alpha^k}z_n(\alpha, t)$ is a linear combination of terms of the form 
\begin{equation*}
f_0^{(k_1)}(\alpha)...f_0^{(k_l)}(\alpha) [(W_n \cdot \nabla)^{k - \sum_i k_i }v_n] (z_n(\alpha, t), t),
\end{equation*}
with $\sum_{i=1}^l k_i \leq k - 1$. The conclusion follows. 
\end{proof}

\section{Uniform \texorpdfstring{$C^{k, \, \mu}$}{C(k, mu)}  Estimates for the Boundaries of the Approximate Patches}

In this section, we obtain the uniform bounds which will allow us to show convergence. We begin by proving a lemma which will have the same role as that of lemma \ref{lemma3.3}.

\begin{lemma} \label{lemma6.1}
Let $m \geq 2$ and $f_1, ... , f_m: \mathbb{R}^N \rightarrow \mathbb{R}$ be $m$ functions, respectively, in $C^{0, \gamma_i}$, with $\gamma_i \in (0,1)$ and $\sum \gamma_i > m - 1$. Let $K_m: \mathbb{R}^N \rightarrow \mathbb{R}$ be a kernel which is smooth away from the origin and homogeneous of degree $-N - m + 1$, and let $\omega \in L^\infty(\mathbb{R}^N)$. Define $\mu = \sum \gamma_i - m + 1$ and
\begin{equation*}
G_m(x) = \int_{\mathbb{R}^N} (f_1(x) - f_1(y))...(f_m(x) - f_m(y))K_m(x - y) \omega(y) dy.
\end{equation*}
Then, there exists a constant $C$, depending only on $K_m$, $\gamma_1,..., \gamma_m$ and the dimension of the space, such that
\begin{equation} \label{eq6.1}
|G_m|_{L^\infty} \leq C |f_1|_{C^{0, \, \gamma_1}}...|f_m|_{C^{0, \, \gamma_m}}|\omega|_{L^\infty},
\end{equation} 
and 
\begin{equation} \label{eq6.2}
|G_m|_{\mu} \leq C |f_1|_{\gamma_1}...|f_m|_{\gamma_m}|\omega|_{L^\infty}.
\end{equation}
\end{lemma}

\begin{proof}
We first show \eqref{eq6.1} (and, implicitly, that $G_m$ is well-defined). Let $R > 0$.
\begin{eqnarray*}
|G_m(x)| & \leq & \int_{\mathbb{R}^N}|f_1(x) - f_1(y)|...|f_m(x) - f_m(y)||K_m(x-y)||\omega(y)| dy \\ 
& = & \int_{|x-y| < R}|f_1(x) - f_1(y)|...|f_m(x) - f_m(y)||K_m(x-y)||\omega(y)| dy + \\ 
& & + \int_{|x-y| \geq R}|f_1(x) - f_1(y)|...|f_m(x) - f_m(y)||K_m(x-y)||\omega(y)| dy \\ 
& = & T_1 + T_2. 
\end{eqnarray*}
For the first term, we use the fact that $|K_m(x)| \leq \frac{C}{|x|^{N+m -1}}$:
\begin{eqnarray*}
T_1 \leq C |f_1|_{\gamma_1}...|f_m|_{\gamma_m} |\omega|_{L^\infty} \int_{|y| \leq R} |y|^{\sum \gamma_i + 1 - m - N} dy.
\end{eqnarray*}
Since $\sum \gamma_i + 1 - m > 0$, this last integral is finite. For the second term, we have 
\begin{eqnarray*}
T_2 \leq C|f_1|_{L^\infty}...|f_m|_{L^\infty}|\omega|_{L^\infty} \int_{|y|\geq R} |y|^{-N - m + 1} dy,
\end{eqnarray*}
and the integral is finite since $m > 1$. The first claim follows. 

To prove \eqref{eq6.2}, let $h \in \mathbb{R}^N$, $h \neq 0$. Then, 
\begin{eqnarray*}
G_m(x+h) - G_m(x)  &=& \int_{|x - y| < 2|h|} (f_1(x+h) - f_1(y))...(f_m(x+h) - f_m(y))\\&& K_m(x+h - y)\omega(y) dy - \\ 
&& - \int_{|x - y| < 2|h|} (f_1(x) - f_1(y))...(f_m(x) - f_m(y)) K_m(x- y)\omega(y) dy + \\ 
&& + \sum_{i = 1}^m \int_{|x-y| \geq 2|h|} (f_1(x) - f_1(y))...(f_{i-1}(x) - f_{i-1}(y))(f_i(x+h) - f_i(x)) \\&&(f_{i+1}(x+h) - f_{i+1}(y))...(f_n(x+h)-f_n(y))K_m(x+h-y) \omega(y) dy + \\ 
&& + \int_{|x-y| \geq 2|h|} (f_1(x)-f_1(y))...(f_m(x)-f_m(y))\\&&(K_m(x+h-y)-K_m(x-y))\omega(y) dy \\ 
& = & S_1 + S_2 + \sum_{i=1}^m R_i + U. 
\end{eqnarray*}
For $S_1$ and $S_2$ we argue as for $T_1$ above to obtain 
\begin{equation*}
|S_1| \leq C |f_1|_{\gamma_1}...|f_m|_{\gamma_m} |\omega|_{L^\infty} \int_{|y| \leq 2|h|} |y|^{\sum \gamma_i + 1 - m - N} dy \leq C |f_1|_{\gamma_1}...|f_m|_{\gamma_m} |\omega|_{L^\infty} |h|^\mu.
\end{equation*} 
For the terms in the sum: 
\begin{eqnarray*}
|R_i| &\leq & C|f_1|_{\gamma_1}...|f_m|_{\gamma_m} |\omega|_{L^\infty} |h|^{\gamma_i} \int_{|x-y| \geq 2|h|} \frac{|x-y|^{\sum_{j=1}^{i-1} \gamma_j}|x+h - y|^{\sum_{j=i+1}^m \gamma_j}}{|x+h-y|^{N+m-1}}dy \\
&\leq & C|f_1|_{\gamma_1}...|f_m|_{\gamma_m}|\omega|_{L^\infty} |h|^{\gamma_i} \int_{|y|>|h|}|y|^{\sum \gamma_j - \gamma_i - (m-1) - N} dy \\
&\leq & C|f_1|_{\gamma_1}...|f_m|_{\gamma_m}|\omega|_{L^\infty} |h|^{\mu}.
\end{eqnarray*}
Finally, for $U$, we use the mean-value inequality and argue as above, while noting that $|\nabla K_m(x)| \leq \frac{C}{|x|^{N+m}}$. The conclusion follows. 
\end{proof}

The following proposition can be regarded as an extension of proposition \ref{prop3.3}. If $u:\mathbb{R}^2 \rightarrow \mathbb{R}^2$ is a vector field, we use the notation $u^i$, $i = 1,2$, to denote the $i^{\text{th}}$ component. 

\begin{proposition} \label{prop6.1}
Assume $\phi_0$ is a $C^{k, \, \mu}$ initial scalar. Let $(\phi^n, v_n)$ be a solution given by proposition \ref{prop2.1} and denote $W_n = \nabla^\perp \phi^n$. Then, for $1 \leq j \leq k$, $(W_n\cdot \nabla)^j v_n(x)$ can be written as a linear combination of terms of the form 
\begin{eqnarray} \label{eq6.3}
\sum_{i_1,...,i_m = 1}^2 \int_{\mathbb{R}^2}&& [(W_n\cdot \nabla)^{j_1-1}W_n(x) - (W_n \cdot \nabla)^{j_1-1} W_n(y)]^{i_1}... \\ &&...[(W_n\cdot \nabla)^{j_m - 1}W_n(x) - (W_n \cdot \nabla)^{j_m -1 } W_n(y)]^{i_m}\partial_{i_1}...\partial_{i_m}K(x-y)\omega_n(y) dy, \nonumber
\end{eqnarray} 
where $K(x) = \frac{1}{2\pi} |x|^{-2} \begin{bmatrix}
-x_2 \\ x_1
\end{bmatrix} $ and $\sum_{l=1}^m j_l = j$, $j_l \geq 1$.
\end{proposition}

\begin{proof}
For ease of exposition, we will denote $u_j = (W_n \cdot \nabla)^{j-1} W_n$ and use the Einstein summation convention. Before we begin the proof, let us also note that the terms of the form \eqref{eq6.3} are well-defined. Indeed, if the number of terms $(W_n \cdot \nabla)^{j_l - 1}W_n$ in the integrand is $m > 1$, then each of these terms have $j_l - 1 \leq j - 2 \leq k-2$, so, since by proposition \ref{prop2.1} $W_n \in C^{k-1, \, \mu}$, each of these terms are in $C^{0, \, \gamma}$ for all $\gamma \in (0,1)$, and lemma \ref{lemma6.1} applies.  For $m =1$, see the remark after lemma \ref{lemma3.3}. 

For $j = 1$, the claim is simply proposition \ref{prop3.3}. We argue inductively. Suppose the claim holds for $(W_n \cdot \nabla)^jv_n$. Since $W_n \cdot \nabla$ is a linear operator, it suffices to argue for one term of the form \eqref{eq6.3}. 

Let
\begin{eqnarray*}
F_j(x) & = & \int_{\mathbb{R}^2} [u_{j_1}^{i_1}(x) - u_{j_1}^{i_1}(y)]...[u_{j_m}^{i_m}(x) - u_{j_m}^{i_m}(y)]\partial_{i_1}...\partial_{i_m}K(x-y)\omega_n(y) dy \\
& = & \int_{\mathbb{R}^2} [u_{j_1}^{i_1}(x) - u_{j_1}^{i_1}(x + y)]...[u_{j_m}^{i_m}(x) - u_{j_m}^{i_m}(x+ y)]\partial_{i_1}...\partial_{i_m}K(-y)\omega_n(x+y) dy,
\end{eqnarray*}
with $\sum j_l = j$. 
Now, using the dominated convergence theorem, we obtain 
\begin{eqnarray*}
\partial_i F_j (x)& =& \int_{\mathbb{R}^2} [u_{j_1}^{i_1}(x) - u_{j_1}^{i_1}(x + y)]...[u_{j_m}^{i_m}(x) - u_{j_m}^{i_m}(x+ y)]\partial_{i_1}...\partial_{i_m}K(-y)\partial_i \omega_n(x+y) dy + \\ 
&& + \sum_{l=1}^m \int_{\mathbb{R}^2} [u_{j_1}^{i_1}(x) - u_{j_1}^{i_1}(x + y)]...[\partial_i u_{j_l}^{i_l}(x) - \partial_i u_{j_l}^{i_l}(x + y)]...[u_{j_m}^{i_m}(x) - u_{j_m}^{i_m}(x+ y)] \\ 
&& \, \, \, \, \,\, \, \, \, \, \,\,\,\,\,\,\,\,\,\,\,\,\,\,\,\,\,\,\, \partial_{i_1}...\partial_{i_m}K(-y)\omega_n(x+y) dy \\ 
& = &\int_{\mathbb{R}^2} [u_{j_1}^{i_1}(x) - u_{j_1}^{i_1}(y)]...[u_{j_m}^{i_m}(x) - u_{j_m}^{i_m}(y)]\partial_{i_1}...\partial_{i_m}K(x-y)\partial_i \omega_n(y) dy + \\ 
&& + \sum_{l=1}^m \int_{\mathbb{R}^2} [u_{j_1}^{i_1}(x) - u_{j_1}^{i_1}(y)]...[\partial_i u_{j_l}^{i_l}(x) - \partial_i u_{j_l}^{i_l}(y)]...[u_{j_m}^{i_m}(x) - u_{j_m}^{i_m}(y)] \\ 
&& \, \, \, \, \,\, \, \, \, \, \,\,\,\,\,\,\,\,\,\,\,\,\,\,\,\,\,\,\, \partial_{i_1}...\partial_{i_m}K(x-y)\omega_n(y) dy.
\end{eqnarray*}
Using also the fact that $W_n \cdot \nabla \omega_n = 0$, we get
\begin{eqnarray*}
W_n^i(x)\partial_i F_j(x) & = & \int_{\mathbb{R}^2} [W_n^i(x) - W_n^i(y)] [u_{j_1}^{i_1}(x) - u_{j_1}^{i_1}(y)]...[u_{j_m}^{i_m}(x) - u_{j_m}^{i_m}(y)]\\
&& \, \, \, \, \,\, \, \, \, \, \,\,\,\,\,\,\,\,\,\,\,\,\,\,\,\,\,\,\,\partial_{i_1}...\partial_{i_m} K(x-y)\partial_i \omega_n(y) dy + \\ 
&& + \sum_{l=1}^m \int_{\mathbb{R}^2} W_n^i(x)[u_{j_1}^{i_1}(x) - u_{j_1}^{i_1}(y)]...[\partial_i u_{j_l}^{i_l}(x) - \partial_i u_{j_l}^{i_l}(y)]...[u_{j_m}^{i_m}(x) - u_{j_m}^{i_m}(y)] \\ 
&& \, \, \, \, \,\, \, \, \, \, \,\,\,\,\,\,\,\,\,\,\,\,\,\,\,\,\,\,\, \partial_{i_1}...\partial_{i_m}K(x-y)\omega_n(y) dy.
\end{eqnarray*}
For the first term, we use Green's theorem, dominated convergence and the fact that $W_n$ is divergence-free to write 
\begin{eqnarray*}
T_1 & = & \lim_{\epsilon \rightarrow 0} \sum_{l = 1}^m \int_{|x-y| \geq \epsilon} [W_n^i(x)-W_n^i(y)][u_{j_1}^{i_1}(x) - u_{j_1}^{i_1}(y)]...[ \partial_i u_{j_l}^{i_l}(y)]...[u_{j_m}^{i_m}(x) - u_{j_m}^{i_m}(y)] \\ 
&& \, \, \, \, \,\, \, \, \, \, \,\,\,\,\,\,\,\,\,\,\,\,\,\,\,\,\,\,\, \partial_{i_1}...\partial_{i_m}K(x-y)\omega_n(y) dy + \\
&& +  \int_{|x-y| \geq \epsilon} [W_n^i(x) - W_n^i(y)] [u_{j_1}^{i_1}(x) - u_{j_1}^{i_1}(y)]...[u_{j_m}^{i_m}(x) - u_{j_m}^{i_m}(y)]\\
&& \, \, \, \, \,\, \, \, \, \, \,\,\,\,\,\,\,\,\,\,\,\,\,\,\,\,\,\,\,\partial_i  \partial_{i_1}...\partial_{i_m} K(x-y)\omega_n(y) dy + \\
&& +  \int_{|x-y|=\epsilon} [W_n^i(x) - W_n^i(y)] [u_{j_1}^{i_1}(x) - u_{j_1}^{i_1}(y)]...[u_{j_m}^{i_m}(x) - u_{j_m}^{i_m}(y)]\\
&& \, \, \, \, \,\, \, \, \, \, \,\,\,\,\,\,\,\,\,\,\,\,\,\,\,\,\,\,\,\partial_{i_1}...\partial_{i_m} K(x-y) \omega_n(y) \frac{x_i - y_i}{|x-y|}dS(y).
\end{eqnarray*}
The integrand in the final term can be bounded by $C |x-y|^{\gamma + \sum_{l=1}^m \gamma_l - 1 - m}$, for any $\gamma, \gamma_1, ..., \gamma_m \in (0,1)$ (arguing as in the proof of lemma \ref{lemma6.1} and noting that $\partial_{i_1}...\partial_{i_m}K(x)$ is homogeneous of degree $- m - 1$). Since the length of the circle of radius $\epsilon$ is also proportional to $\epsilon$, it suffices to choose the coefficients such that $\gamma + \sum \gamma_i > m$ in order to see that the term vanishes. So, 
\begin{eqnarray*}
T_1 & = &  \sum_{l = 1}^m \int_{\mathbb{R}^2} [W_n^i(x)-W_n^i(y)][u_{j_1}^{i_1}(x) - u_{j_1}^{i_1}(y)]...[ \partial_i u_{j_l}^{i_l}(y)]...[u_{j_m}^{i_m}(x) - u_{j_m}^{i_m}(y)] \\ 
&& \, \, \, \, \,\, \, \, \, \, \,\,\,\,\,\,\,\,\,\,\,\,\,\,\,\,\,\,\, \partial_{i_1}...\partial_{i_m}K(x-y)\omega_n(y) dy + \\
&& +  \int_{\mathbb{R}^2} [W_n^i(x) - W_n^i(y)] [u_{j_1}^{i_1}(x) - u_{j_1}^{i_1}(y)]...[u_{j_m}^{i_m}(x) - u_{j_m}^{i_m}(y)]\\
&& \, \, \, \, \,\, \, \, \, \, \,\,\,\,\,\,\,\,\,\,\,\,\,\,\,\,\,\,\,\partial_i  \partial_{i_1}...\partial_{i_m} K(x-y)\omega_n(y) dy.
\end{eqnarray*}
Note that the last term above has the wanted form for $j+1$. Plugging $T_1$ back into the expression for $W^i \partial_i F_j$ and noting that 
\begin{equation*}
W_n^i(x)[\partial_i u_{j_l}^{i_l}(x) - \partial_i u_{j_l}^{i_l}(y)] + [W_n^i(x)-W_n^i(y)]\partial_i u_{j_l}^{i_l}(y) =  u_{j_l + 1}^{i_l}(x) - u_{j_l + 1}^{i_l}(y)
\end{equation*} 
shows that each term in the sum also has the required form, concluding the proof.
\end{proof}

We are now ready to prove the wanted uniform bounds. The arguments are essentially those we have already used in proposition \ref{prop3.4}. 

\begin{proposition} \label{prop6.2}
Let $\phi_0$ be a $C^{k, \, \mu}$ initial scalar. Let $(\phi^n, v_n)$ be a solution given by proposition \ref{prop2.1}, and $n \geq M$ with $M$ sufficiently large such that corollary \ref{cor3.3} holds. Fix $T \in \mathbb{R}_+$. Then, there exists a constant $C(T)$ depending on $T$, $\mu$, $k$, $\omega_0$ and $\phi_0$, but not on $t$ or $n$ such that for all $t \in [-T, T]$,
\begin{equation*}
|(W_n\cdot \nabla)^{k-1} W_n(\cdot, t)|_{C^{0, \, \mu}} \leq C(T). 
\end{equation*}
\end{proposition}

\begin{proof}
We argue inductively. The case $k = 1$ is the subject of corollary \ref{cor3.3}. Suppose the claim holds for $k$, we show it holds for $k+1$. Since $\phi_0 \in C^{k+1, \, \mu}$, it follows that $\phi_0 \in C^{j, \, \gamma}$ for all $j \leq k$ and $\gamma \in (0,1)$. Therefore, we can choose $\gamma$ such that $\mu = m \gamma - m + 1$, and the conditions of lemma \ref{lemma6.1} are satisfied. By the inductive hypothesis, there exists a constant such that for all such $j$ and $t \in [-T, T]$:
\begin{equation*}
|(W_n \cdot \nabla)^{j} W_n|_{C^{0, \, \gamma}} \leq C. 
\end{equation*}
We now argue precisely as in proposition \ref{prop3.4} to obtain 
\begin{eqnarray*}
|(W_n\cdot \nabla)^k W_n(x, t) - (W_n\cdot \nabla)^k W_n(y, t)| \leq  |(W_0 \cdot \nabla)^k W_0|_{\mu} \exp \bigg[\mu \int_0^t |\nabla v_n(\cdot, s)|_{L^\infty} ds \bigg] |x - y|^\mu \\
 + \int_0^t |(W_n \cdot \nabla)^{k+1} v_n (\cdot, \tau)|_{\mu} \exp \bigg[\mu \int_\tau^t |\nabla v_n(\cdot, s)|_{L^\infty} ds \bigg]|x-y|^\mu d\tau. 
\end{eqnarray*}
Using proposition \ref{prop6.1}, together with lemma \ref{lemma6.1} for lower order terms and lemma \ref{lemma3.3} for the highest order one, we obtain that for all $t \in [-T,T]$,
\begin{eqnarray*}
|(W_n \cdot \nabla)^{k+1} v_n(\cdot, t)|_\mu & \leq & C_0 |(W_n \cdot \nabla)^k W_n(\cdot, t)|_{\mu} (|\nabla v_n(\cdot, t)|_{L^\infty} + |\omega_0|) + C(T). 
\end{eqnarray*}
Therefore, 
\begin{eqnarray*}
|(W_n \cdot \nabla)^k W_n (\cdot, t)|_{\mu} \leq  |(W_0 \cdot \nabla)^k W_0|_\mu \exp \bigg[\mu \int_0^t |\nabla v_n(\cdot, s)|_{L^\infty} ds \bigg] + \\ 
  + C_0 \int_0^t (|\nabla v_n(\cdot, \tau)|_{L^\infty} + |\omega_0|)|(W_n \cdot \nabla)^{k} W_n (\cdot, \tau)|_{\mu} \exp \bigg[\mu \int_\tau^t |\nabla v_n(\cdot, s)|_{L^\infty} ds \bigg] d\tau +\\
  + C(T) t \exp \bigg[\mu \int_0^t |\nabla v_n(\cdot, s)|_{L^\infty} ds \bigg].
\end{eqnarray*}
Multiplying both sides by $\exp \bigg[- \mu \int_0^t | \nabla v_n(\cdot, s)|_{L^\infty} ds \bigg]$ and denoting the left-hand side by $G(t)$ we obtain:
\begin{equation*}
G(t) \leq G(0) + C(T) t + C_0 \int_0^t (|\nabla v_n(\cdot, \tau)|_{L^\infty} + |\omega_0|) G(\tau) d\tau .
\end{equation*}
Applying, now, Gr\"onwall's lemma, 
\begin{equation*}
|(W_n \cdot \nabla)^k W_n(\cdot, t)|_\mu \leq [|(W_0 \cdot \nabla)^k W_0|_\mu + C(T)t] \exp \bigg[(C_0 + \mu) \int_0^t |\nabla v_n(\cdot, s)|_{L^\infty} ds + C_0 |\omega_0| t \bigg].
\end{equation*}
Since by corollary \ref{cor3.3}, $|\nabla v_n(\cdot, t)|_L^\infty$ are also uniformly bounded on $[-T, T]$, the conclusion follows for the H\"older semi-norms. 

It remains to show that the $L^\infty$ norms are also uniformly controlled.  It follows from \eqref{eq5.4} that
\begin{equation*}
|(W_n \cdot \nabla)^{k} W_n(\cdot, t)|_{L^\infty} \leq |(W_0 \cdot \nabla)^k W_0|_{L^\infty} + \int_0^t |(W_n \cdot \nabla)^{k+1}v_n(\cdot, s)|_{L^\infty} ds.
\end{equation*}
Using now again proposition \ref{prop6.1}, together with the $L^\infty$ estimates of lemmas \ref{lemma6.1} and \ref{lemma3.3}, 
\begin{equation*}
|(W_n \cdot \nabla)^{k+1} v_n(\cdot, s)|_{L^\infty} \leq C|(W_n \cdot \nabla)^{k} W_n(\cdot, s)|_{L^\infty} + C(T),
\end{equation*}
where we also used the fact that the H\"older semi-norms are bounded uniformly on the interval of interest. So, 
\begin{equation*}
|(W_n \cdot \nabla)^{k} W_n(\cdot, t)|_{L^\infty} \leq |(W_0 \cdot \nabla)^k W_0|_{L^\infty}  + C(T) t + C \int_0^t|(W_n \cdot \nabla)^{k} W_n(\cdot, s)|_{L^\infty} ds.
\end{equation*}
Gr\"onwall's lemma concludes the proof. 
\end{proof} 

\begin{corollary} \label{cor6.1}
Assume $\phi_0$ is a $C^{k, \, \mu}$ initial scalar. Let $\phi^n$ be solutions given by proposition \ref{prop2.1}, and $z_n:\mathbb{S}^1 \times \mathbb{R} \rightarrow \mathbb{R}^2$ the associated parametrizations of the boundaries, with $n \in \mathbb{N}$ sufficiently large such that the requirements of proposition \ref{prop6.2} are satisfied. For each $T \in \mathbb{R}_+$, there exists a constant $C(T)$, independent of $n$ or $t$, such that for all $t \in [-T, T]$,
\begin{equation*}
|z_n(\cdot, t)|_{C^{k, \, \mu}(\mathbb{S}^1)} \leq C(T),
\end{equation*}
and
\begin{equation*}
\bigg|\frac{d}{dt} \frac{d^k}{d\alpha^k}z_n(\cdot, t)\bigg|_{L^\infty(\mathbb{S}^1)} \leq C(T). 
\end{equation*}
\end{corollary}

\begin{proof}
The first claim follows from proposition \ref{prop6.2} and corollary \ref{cor5.1}. The second, from propositions \ref{prop6.1} and \ref{prop6.2}, lemmas \ref{lemma6.1} and \ref{lemma3.3}, and corollary \ref{cor5.2}. 
\end{proof}

\section{\texorpdfstring{$C^{k, \, \mu}$}{C(k,mu)} Regularity for the Boundary of the Vortex Patch}

In this section, we use the uniform bounds together with the Arzel\`a-Ascoli theorem to show that there exists a subsequence of curves $\{z_n\}$ which converges to the boundary $z$ of the the patch, thus implying its regularity. 

\begin{proposition} \label{prop7.1}
Let $\phi:\mathbb{R}^2 \times \mathbb{R} \rightarrow \mathbb{R}$ be a solution of the vortex patch equations with $C^{k, \, \mu}$ initial scalar $\phi_0$, as given by theorem \ref{thm1}; and let $z:\mathbb{S}^1 \times \mathbb{R} \rightarrow \mathbb{R}^2 $ be a parametrization of the boundary as in \ref{eq5}.  Then, for each time $t \in \mathbb{R}$, $z(\cdot, t) \in C^{k, \, \mu}(\mathbb{S}^1)$. 
\end{proposition}

\begin{proof}
The argument is inductive. We first consider $k=1$. Let $z_n(\alpha, t)$ be the subsequence of parametrizations of the approximate patches corresponding to the subsequence of $\phi^n$ given by proposition \ref{prop4.1}. Then, since $z_n(\alpha, t) = X_n(z_0(\alpha), t)$, $z(\alpha, t) = X(z_0(\alpha), t)$, and $X_n$ converges uniformly on compact sets of $\mathbb{R}^2 \times \mathbb{R}$ to $X$ (see remark after proposition \ref{prop4.1}), we have that $z_n$ converges uniformly to $z$ on each set $\mathbb{S}^1 \times [-T, T]$, $T \in \mathbb{R}_+$. By corollary \ref{cor6.1}, $\frac{d}{d\alpha} z_n(\cdot, t)$ are uniformly H\"older in the $\alpha$ variable and uniformly bounded. Also, for $t_1, t_2 \in [-T, T]$, by the mean-value theorem, there exists $\tau \in [-T,T]$ such that
\begin{equation*}
\bigg|\frac{dz_n}{d\alpha}(\alpha, t_1) -  \frac{dz_n}{d\alpha}(\alpha, t_2)\bigg| \leq \bigg| \frac{d}{dt} \frac{d}{d\alpha} z_n(\cdot, \tau)  \bigg|_{L^{\infty}(\mathbb{S}^1)} |t_1 - t_2|,
\end{equation*} 
which, shows, by corollary \ref{cor6.1}, that $\frac{d}{d\alpha} z_n$ are also equicontinuous in the time variables. Therefore, by the Arzel\`a-Ascoli theorem and a diagonalization argument, there exists a subsequence (for which we keep the notation) that converges uniformly on all compact sets $\mathbb{S}^1 \times [-T, T]$, $T \in \mathbb{R}_+$ to some mapping $\zeta_1(\alpha, t)$. Note that by construction $\zeta_1(
\cdot, t) \in C^{0, \, \mu}(\mathbb{S}^1)$, for each $t \in \mathbb{R}$. We show that $\frac{d}{d\alpha}z = \zeta_1$. Indeed, it suffices to show that 
\begin{equation*}
z(\alpha_2, t) - z(\alpha_1, t) = \int_{\alpha_1}^{\alpha_2} \zeta_1(s, t) ds. 
\end{equation*}
But we have 
\begin{equation*}
z_n(\alpha_2, t) - z_n(\alpha_1, t) = \int_{\alpha_1}^{\alpha_2} \frac{d}{d\alpha} z_n(s, t) ds.
\end{equation*}
The left-hand side converges to $z(\alpha_2, t) - z(\alpha_1, t)$ as $n \rightarrow \infty$, and the right-hand side to $\int_{\alpha_1}^{\alpha_2} \zeta_1(s, t) ds$ by the uniform convergence of $\frac{d}{d\alpha} z_n$ to $\zeta_1$ on compact sets. This proves the case $k = 1$. 

Suppose that for $k$, we have a subsequence of parametrized curves $z_n$ such that for each $j \leq k$, $\frac{d^j}{d\alpha^j}z_n$ converge uniformly to $\frac{d^j}{d\alpha^j} z$ on compact sets. The argument that follows is exactly the one above: corollary \ref{cor6.1} implies that $\frac{d^{k+1}}{d\alpha^{k+1}}z_n$ converges uniformly on compact sets to some $\zeta_{k+1}$. Then, also $\zeta_{k+1}(\cdot, t) \in C^{0, \mu}$ for all times $t \in \mathbb{R}$. Finally, we check as above that $\zeta_{k+1} = \frac{d^{k+1}}{d\alpha^{k+1}}z$. This proves the inductive hypothesis and concludes the proof.
\end{proof}

Theorem \ref{thm1} (together with the discussion following its proof in section 3) and proposition \ref{prop7.1} imply theorem \ref{thm2}. 

\subsection*{Acknowledgements} This paper was written to satisfy the undergraduate junior independent work requirement. I am very grateful to Professor Peter Constantin for agreeing to direct this paper, for proposing the problem to me and suggesting that I could start by regularizing the Heaviside function in \eqref{eq1.3}, as well as for his guidance and encouragement throughout.

\end{document}